\documentclass[12pt,a4paper]{amsart}
\makeatletter
\renewcommand\normalsize{%
    \@setfontsize\normalsize{9}{14pt plus .3pt minus .3pt}%
    \abovedisplayskip 10\p@ \@plus4\p@ \@minus4\p@
    \abovedisplayshortskip 6\p@ \@plus2\p@
    \belowdisplayshortskip 6\p@ \@plus2\p@
    \belowdisplayskip \abovedisplayskip}
\renewcommand\small{%
    \@setfontsize\small{9.5}{12\p@ plus .2\p@ minus .2\p@}%
    \abovedisplayskip 8.5\p@ \@plus4\p@ \@minus1\p@
    \belowdisplayskip \abovedisplayskip
    \abovedisplayshortskip \abovedisplayskip
    \belowdisplayshortskip \abovedisplayskip}

\ifdefined\pdfpagewidth
\else
\fi
\calclayout
\makeatother

\usepackage{graphicx}
\usepackage[toc,page]{appendix}
\usepackage{chngcntr}
\usepackage{etoolbox}
\usepackage{fullpage}
\usepackage{tikz}

\usepackage[T1]{fontenc} 
\usepackage[utf8]{inputenc} 
\usepackage[english]{babel}
\usepackage{amsmath}
\usepackage{amssymb}
\usepackage{amsthm}
\usepackage{comment}
\usepackage{enumitem}
\usepackage{braket}
\usepackage{graphicx}
\usepackage{sidecap}
\usepackage{hyperref}
\usepackage{tikz-cd}
\usepackage{mathtools}
\DeclarePairedDelimiter\ceil{\lceil}{\rceil}

\definecolor{gold}{rgb}{0.85,0.65,0}

\usepackage[all]{xy}
\usepackage{mathtools}
\usepackage{amsmath,amssymb,tikz-cd,tikz}

\theoremstyle{plain}
\newtheorem{thm}{Theorem}[section]
\newtheorem{lem}[thm]{Lemma}
\newtheorem{prop}[thm]{Proposition}
\newtheorem{cor}[thm]{Corollary}

\newtheorem{conj}[thm]{Conjecture}
\newtheorem{proposal}[thm]{Proposal}

\theoremstyle{definition}
\newtheorem{df}[thm]{Definition}

\theoremstyle{remark}
\newtheorem{example}[thm]{Example}
\newtheorem{rmk}[thm]{Remark}

\def\a{\mathfrak{a}}

\def\pt{\mathrm{pt}}

\def\Im{\mathfrak{Im}}
\def\Re{\mathfrak{Re}}
\def\ch{\mathrm{ch}}
\def\Ch{\mathrm{Ch}}

\def\End{\mathrm{End}}

\def\Z{\mathbb{Z}}
\def\Ker{\mathrm{Ker}}
\def\Stab{\mathrm{Stab}}

\def\H{\mathrm{H}}

\def\GL{\mathrm{GL}}

\def\Hom{\mathrm{Hom}}
\def\Ext{\mathrm{Ext}}
\def\o{\mathcal{O}}
\def\pr{\mathbb{P}}
\def\Q{\mathbb{Q}}
\newcommand{\R}{\mathbb{R}}
\newcommand{\C}{\mathbb{C}}
\newcommand{\ir}{\sqrt{-1}}
\newcommand{\slicing}{{\mathcal{P}}}
\newcommand{\D}{\mathrm{D}}
\newcommand{\A}{\mathcal{A}}
\newcommand{\F}{\mathcal{F}}

\newcommand{\map}{\longrightarrow}

\DeclarePairedDelimiter\norm{\lVert}{\rVert}

\author[V.~Zuliani]{Vanja Zuliani}
\address{\parbox{0.9\textwidth}{Universit\'e Paris-Saclay, CNRS, Laboratoire de Math\'ematiques d'Orsay\\[1pt]
Rue Michel Magat, B\^at. 307, 91405 Orsay, France
\vspace{1mm}}}
\email{{vanja.zuliani@universite-paris-saclay.fr}}

\title{Semiorthogonal decompositions of projective spaces from small quantum cohomology}

\begin{document}
\begin{abstract}
    In a recent article Halpern-Leistner defines the notion of quasi-convergent path in the space of Bridgeland stability conditions. Such a path induces a semiorthogonal decomposition of the derived category. 
    We investigate quasi-convergent paths in the stability manifold of projective spaces  and answer positively to two questions posed by Halpern-Leistner. 
    We construct  quasi-convergent paths that start from the geometric region of the stability space and whose central charge is given by a fundamental solution of the quantum differential equation.
    We also construct quasi-convergent paths whose central charges are the quantum cohomology central charges defined by Iritani.
\end{abstract}
\maketitle

\tableofcontents

\section{Introduction}

Let $X$ be a smooth projective variety over the complex numbers and $\mathrm{D^b}(X)$ the bounded derived category of coherent sheaves on $X$.
The goal of this note is to partially clarify the relations between the following mathematical objects:

\begin{itemize}
    \item [1] the semiorthogonal decompositions (SODs) of $\mathrm{D^b}(X)$,
    \item [2] the quantum cohomology of $X$,
    \item [3] the paths in the space of stability conditions $\Stab(\mathrm{D^b}(X))$.
\end{itemize}
We focus on the example of projective spaces $X=\pr^{N-1}$.
Before stating our results let us recall some general definitions and conjectures.

In \cite{Ballard_2015,Diemer_2016} the authors investigate homological mirror symmetry for toric varieties. 
For a projective toric variety $T$ with mild singularities they show that certain variations of the Landau--Ginzburg model of $T$ give  semiorthogonal decompositions of the Fukaya--Seidel category; by homological mirror symmetry this SOD should correspond to a SOD of $\mathrm{D^b}(T)$.
In \cite{DHL23noncomm} the degeneration of the Landau--Ginzburg model is reinterpreted as a path in the space of Bridgeland stability conditions.
This reinterpretation leads to a general definition of \text{quasi-convergent path} 
\begin{equation*}
\begin{split}
    \sigma_{\bullet} : (0,\infty)&\to \Stab(\mathrm{D^b}(X))\\
    r&\mapsto \sigma_r=(Z_r,\slicing_r) 
\end{split}
\end{equation*} 
see Definition~\ref{def_quaiconpath}.
The crucial property of a quasi-convergent path is that it gives a SOD of $\mathrm{D^b}(X)$, see Proposition~\ref{Proposition_2.20_HLJR}.

In order to investigate rationality questions in birational geometry, it would be helpful to have some sort of distinguished SOD. 
Kuznetsov observes that there should be a way to define a birational invariant of a variety $X$ using the SODs of $\mathrm{D^b}(X)$; he introduces the notion of Griffiths component, shows that it is not well defined and he gives some insights on how to fix this failure, see \cite{Kuznetsov_2016_rationality,kuznetsov2013simple}.
The Griffiths components are the "big" components in a SOD of $\mathrm{D^b}(X)$ that cannot be refined.
If the number of Griffiths components in the SODs of $\mathrm{D^b}(X)$ is well defined then it is a birational invariant and for instance, it would apply to the problem of rationality of cubic hypersurfaces.
The derived category of projective spaces have full exceptional collections, so they have no Griffiths components.
We could redefine the Griffiths components by considering only the SODs that are given by some special kind of quasi-convergent paths and that cannot be refined.
In particular the components of an SOD given by a quasi-convergent path have a stability condition.
We also note that in order to prove that the Griffiths component is not well defined, Kuznetsov exhibits the derived category of a rational threefold $R$ and a SOD of $\mathrm{\D^b}(R)$ with a component $\A$ that is "big" and it does not admit any Serre invariant stability condition, see \cite{sung2022remarks, alex2014semiorthogonal,Kuznetsov_2009,Kuznetsov_2016_rationality}.
Conjecturally $\A$ does not admit any stability condition and so in particular it could not appear as a component in a SOD given by a quasi-convergent path.

As we already noted a distinguished quasi-convergent path would give a distinguished SOD.
Halpern-Leistner asks whether there are quasi-convergent paths $\sigma_{\bullet}=(Z_{\bullet},\slicing_{\bullet})$ where $Z_{\bullet}$ is the quantum cohomology central charge defined by Iritani in \cite{Iritani_2009}.
For a smooth Fano variety $F$ and 
$ B\subseteq \H^{\bullet}(F,\C)$ the domain of convergence of the Gromov--Witten potential we consider the meromorphic Dubrovin connection $\Tilde{\nabla}$ on
the trivial vector bundle 
\begin{equation*}
    \H^{\bullet}(F,\C)\times(B\times \pr^1)\to B\times \pr^1.
\end{equation*}
The connection $\Tilde{\nabla}$ is regular on $B\times\C^*$ and for each $\tau\in B$ there is a canonical fundamental solution of flat sections of $\Tilde{\nabla}_{|_{\tau}}$ denoted as follows
\begin{equation*}
    \Phi^{\tau} : \C^*\to \End(\H^{\bullet}(F,\C))
\end{equation*}
see \cite{dubrovin1994geometry,dubrovin1998painleve,Galkin_2016}.
We recall that the Gamma class $\hat{\Gamma}$ of $F$ is
\begin{equation*}
    \hat{\Gamma}:=\Pi_i\Gamma(1+\delta_i) \in \H^{\bullet}(F,\R)
\end{equation*}
where $\delta_i$ are the Chern roots of the tangent bundle of $F$ and $\Gamma$ is the Euler's Gamma function.
The \emph{quantum cohomology central charge} is defined in general as 
\begin{equation}\label{quantum_central_charge_intro}
    Z(V):=(2\pi z)^{\frac{\dim F}{2}}\int_{F}\Phi^0(\hat{\Gamma}\cdot\Ch(V)) \text{ where } \Ch:=(2\pi\ir)^{\deg/2}\ch
\end{equation}
for any $V\in \mathrm{D^b}(F)$, see \cite[Paragraph 3.8]{Galkin_2016}{}.
We state a particular case of \cite[Proposal III]{DHL23noncomm}{}.
\begin{proposal}[{\cite[Proposal III]{DHL23noncomm}}]\label{proposalIII_DHL_intro}
    Let us consider a smooth Fano variety $F$.
    Then there exists a fundamental solution $\Phi'$ of $\Tilde{\nabla}_{|_{\tau}}$ for some $\tau\in B$ and a quasi-convergent path
    $\sigma_r=(Z', \slicing_r),$ $r\in\R_{>0}$ in $\Stab(\mathrm{D^b}(F))$ with
    \begin{equation*}
    Z'(\alpha):=\int_{F}\Phi'(\alpha)
    \end{equation*}
    for $\alpha\in \H^{\bullet}(F,\C)$.
    We omit here a technical condition, for a complete statement see Proposal~\ref{proposalIII_DHL}.
\end{proposal}
In \cite[Remark 13]{DHL23noncomm} the proposed candidate for $Z'$ is the quantum cohomology central charge~\eqref{quantum_central_charge_intro}.
We observe that we can generalize the notion of quantum cohomology central charge simply by replacing $\Phi^0$ in~\eqref{quantum_central_charge_intro} with $\Phi^{\tau}$, for any $\tau\in B$.
So, by deforming the quantum cohomology central charge, we state the following 

\begin{conj}\label{conj_sigma_from_qcoh_intro}
    Let us consider a Fano variety $F$, $\tau\in  \H^{2}(F,\C)$, $\Phi^{\tau}$ the fundamental solution of $\Tilde{\nabla}_{|_{\tau}}$ and the induced isomorphism
    \begin{equation*}\label{eq_iso_Phi_tau_intro}
\begin{split}
    \Phi^{\tau} : \H^{\bullet}(F)&\overset{\cong}{\longrightarrow} \left\{ s:\R_{> 0}\to \H^{\bullet}(F): \Tilde{\nabla}_{|_{\tau}} s=0   \right\}\\
    \alpha&\longmapsto \Phi^{\tau}\alpha
\end{split}
\end{equation*}
see Proposition~\ref{prop_GGI2.3.1} and Remark~\ref{rmk_can_sol_fortau_non_0}.
    Then there exists a quasi-convergent path
    $\sigma_r=(Z^{\tau}, \slicing_r),$ $r\in\R_{>0}$ with
    \begin{equation*}
        Z^{\tau}(V):=(2\pi r)^{\frac{\mathrm{dim}(F)}{2}}\int_{F}\Phi^{\tau}(\hat{\Gamma}\cdot\Ch(V)).
    \end{equation*}
\end{conj}
We prove the conjecture for all projective spaces, see Theorem~\ref{thm_conj_pr}.
For the projective plane $\pr^2$ we prove that there is a quasi-convergent path as in Conjecture~\ref{conj_sigma_from_qcoh_intro} such that the induced SOD is $\{\o,\o(1),\o(2)\}$, see Theorem~\ref{thm_conj_P2}.

There are two standard constructions of Bridgeland  stability conditions: 
one uses the notion of Gieseker slope and iteration of tilting and the other uses a full exceptional collection of the derived category.
The first produces stability conditions that are usually called geometric and the skyscrapers sheaves of points are stable with respect to those stability conditions. 
We will use the terminology "geometric stability condition" for  a stability condition for which all skyscraper sheaves of points are stable of the same phase.
The second construction produces stability conditions that are called algebraic stability conditions. 
Algebraic stability conditions were constructed in \cite{Macri07StabCurves}, our Theorem~\ref{thm_stabcond_fromexc} is a refined version of the original construction. 
It is conjectured that there are quasi-convergent paths that start in the geometric part of the stability manifold and end in the algebraic part.
If we had a good description of such paths we could start from a geometric part of the stability manifold and classify the semiorthogonal decompositions whose components admits stability conditions.
In \cite[Remark 14]{DHL23noncomm}{} it is asked whether quasi-convergent paths arising in Proposal~\ref{proposalIII_DHL_intro}  start in the geometric part and end in the algebraic part.
In Theorem~\ref{thm_proposalIII_proj} and Corollary~\ref{cor_path_starts_in_geom_proj} we construct 
quasi-convergent paths in the stability manifolds of projective spaces that satisfies such properties, so we fulfill Proposal~\ref{proposalIII_DHL_intro} and we answer positively the question in \cite[Remark 14]{DHL23noncomm}{}. 

Let us recall some known results: Proposal~\ref{proposalIII_DHL_intro} for curves is fulfilled, see \cite{DHL23noncomm}, moreover \cite[Proposition 25]{DHL23noncomm}{} shows that for all Fano varieties that satisfies \cite[Gamma conjecture II]{Galkin_2016}{} the Proposal~\ref{proposalIII_DHL_intro} is true.
Our Theorem~\ref{thm_gammaII_proposalIII} is a refinement of \cite[Proposition 25]{DHL23noncomm}{}, it permits us to conclude that in some specific cases the path of stability conditions starts in the geometric region.

\subsection*{Acknowledgments}
I would like to thank my advisor  Emanuele Macr{\`i} for suggesting me this problem and for the support  during the work and the writing of this article. 
I would like to also thank Veronica Fantini and Daniel Halpern-Leistner for helpful conversations. 
Moreover I thank Zhirui Li and the anonymous referee for pointing out some typos in a previous version of the paper.
The author was supported by the ERC Synergy Grant 854361 HyperK.

\section{Quasi-convergent paths}

In this section we review  some results on the correspondence between certain paths in the space of stability conditions and  semiorthogonal decompositions of the derived category, see \cite[Section 2]{leistner2023stability}{}.

\subsection{Bridgeland stability conditions}
Let us denote by $\D$ a triangulated category.
We will denote by $K(\D)$ the Grothendieck group of $\D$ and let us fix $\nu : K(\D)\map \Lambda$ a group homomorphism to a finite rank lattice.
We will follow the standard references \cite{2007_stab_cond_on_triang_cat,bayer2019short}.
\begin{df}[{\cite[Bridgeland stability condition]{2007_stab_cond_on_triang_cat,bayer2019short}{}}]\label{def_stab_cond}
    A \emph{Bridgeland stability condition} is a pair $\sigma=(Z,\slicing)$ where 
    $Z:\Lambda \to \C$ is a group homomorphism and
    $\slicing$ is a \emph{slicing} i.e.,  a collection of full subcategories $\slicing(\phi),$ $\phi\in \R$ of $\D$ such that
    \begin{itemize}
        \item [(a)] $\slicing(\phi+1)=\slicing(\phi)[1]$ for all $\phi\in \R$,
        \item [(b)] for $\phi_1<\phi_2$ and $E_i\in \slicing(\phi_i),$ $i=1,2$, we have $\Hom(E_2,E_1)=0$, and
        \item [(c)] for any $0\neq E\in \D$ there is a sequence of morphisms 
        $0=E_0\to E_1\to\cdots\to E_m=E$
        and a sequence of real numbers $\phi_1>\phi_2>\dots>\phi_m$ such that  $A_i:=\mathrm{cone}(E_{i-1}\to E_i)$ is in $\slicing(\phi_i)$ for $i=1,\dots,m$.
    \end{itemize}
    Moreover we ask that for any non zero $E\in \slicing(\phi)$ we have $Z(\nu(E))\in \R_{>0}e^{\pi \ir \phi}$ and that the following \emph{support condition} is satisfied
    \begin{equation*}
        \sup \left\{\frac{\norm{\nu(E)}}{|Z(\nu(E))|}|\text{ for } E \in \bigcup_{\phi\in\R}\slicing(\phi)\setminus\{0\} \right\}<+\infty.
    \end{equation*}
\end{df}

We denote by
$\Stab(\D)$ the set  of (Bridgeland) stability conditions on $\D$; it has a  natural structure of complex manifold of dimension $\mathrm{rk} (\Lambda)$ and the forgetful map $\mathcal{Z}:\Stab(\D)\to \Hom(\Lambda,\C)$ is a local biholomorphism.
We call $0\neq E\in \slicing(\phi)$ a  \emph{semistable} object of phase $\phi$.
For any non zero $E\in\D$ we call the sequence in Definition \autoref{def_stab_cond}{(c)} the \emph{Harder--Narasimhan} (HN) filtration of $E$, the cones $A_j$ are called the \emph{HN factors} and we denote by $\phi^{\pm}(E)$ the largest (resp. smallest) phase of the HN factors of $E$.

We will only treat the case when $\D$ is the bounded derived category of coherent sheaves of a smooth projective variety $X$.
We will choose $\Lambda:=\mathrm{im}(\ch:K(\D)\to \H^{\bullet}(X,\Q)),\nu=\ch$ and we will write $Z(E)$ instead of $Z(\nu(E))$ for $E\in \D$.

\subsection{Quasi convergent paths}
We will denote by $\sigma_{\bullet} : (0,\infty)\to \Stab(\D)$ a continuous map (path) and for any non zero object $E\in \D$ we denote by $\phi_t^{\pm}(E)$ the largest (resp. smallest) phase of the $\sigma_t$-HN filtration of $E$.
We introduce the notation for the mass $m_t(E):=\sum|Z_t(A_i^t)|$ where $A_i^t$ are the $\sigma_t$-HN factors of $E$, see Definition~\ref{def_stab_cond}(c).
We define an average phase $\phi_t(E):=\frac{1}{m_t(E)}\sum \phi_t^{+}(A_i^t)$.

We observe that an object $0\neq E\in \D$ is $\sigma-$semistable if and only if $\phi^+(E)=\phi^-(E)$, so   the natural notion of semistability associated to a path $\sigma_{\bullet}$ is the following one.
\begin{df}[Limit semistable object]
    A non zero object $E\in \D$ is called \emph{limit semistable} (with respect to $\sigma_{\bullet}$) if 
    \begin{equation*}
        \lim_{t\to \infty}(\phi_t^{+}(E)-\phi_t^{-}(E))=0.
    \end{equation*}
\end{df}

\begin{df}[{\cite[Quasi-convergent path]{leistner2023stability}}]\label{def_quaiconpath}
    A path $\sigma_{\bullet}$ in the space of stability conditions is called \emph{quasi-convergent} if:
\begin{enumerate}
    \item [(1)] for any $E\in \D$ there exists a sequence of morphisms
    $0=E_0\to E_1\to E_2 \to \cdots \to E_{n-1}\to E_n=E$ with limit semistable cones $G_i:=\mathrm{cone}(E_{i-1}\to E_i)$ such that 
    \begin{equation*}
        \liminf_{t\to \infty}(\phi_t(G_i)-\phi_t(G_{i+1}))>0
    \end{equation*}
    It is called the \emph{limit semistable filtration}.
    \item [(2)] Given a limit semistable object $E$ we  
    define the average logarithm 
    \begin{equation*}
        l_t(E):=\ln(|Z_t(E)|)+\ir \pi \phi_t(E).
    \end{equation*}
    For two limit semistable objects $E,F$
    we write $l_t(E|F):=l_t(E)-l_t(F)$ and we ask that there exists
    \begin{equation*}
        \lim_{t\to \infty}\frac{l_t(E|F)}{1+|l_t(E|F)|}.
    \end{equation*}
\end{enumerate}
\end{df}
We observe that the the first condition in the definition of quasi-convergent path asks for the existence of a generalized slicing where the order is not given by the order of $\R$, as in the definition of Bridgeland stability conditions, but by the asymptotics of the phases $\phi_t(E)$ as $t\to \infty$, for $E$ limit semistable.
Moreover if a path $\sigma_{\bullet}$ converges to a $\sigma\in\Stab(\D)$ then it is quasi-convergent, the limit semistable objects are the $\sigma-$semistable objects and the limit semistable filtration is the $\sigma-$HN filtration.
This suggests that we can interpret a quasi-convergent path that does not converge to a point in $\Stab(\D)$ as a point in the boundary of a compactification of $\Stab(\D)$.

\subsection{Semiorthogonal decomposition}

One of the most interesting features of a quasi-convergent path $\sigma_{\bullet}$ in $\Stab(\D)$ is that it gives a SOD of $\D$.
In order to understand this point we need to consider some equivalence relation on the set of limit semistable objects:
equivalent  limit semistable objects will generate the components of the SOD.

\begin{df}\label{def_SOD}
    A \emph{semiorthogonal decomposition} (SOD) of $\D$ is an ordered set of full triangulated subcategories $\{\D_1,\dots,\D_n\}$ such that 
    \begin{enumerate}
        \item for each $j<i$ and $E_j\in \D_j, E_i\in \D_i$ we have $\Hom^{\bullet}(E_i,E_j)=0$
        \item for each $E\in \D$ there exists a sequence of morphisms
        \begin{equation*}
            0=E_{n}\to E_{n-1}\to \cdots \to E_1\to E_0=E
        \end{equation*}
        such that $\mathrm{cone}(E_{j}\to E_{j-1})\in \D_j$.
    \end{enumerate}
\end{df}

\begin{df}[{\cite[Definition 2.16]{leistner2023stability}{}}]
    Given two limit semistable objects $E,F$ we will write $E\sim^{\text{inf}} F$ if  
    \begin{equation*}
        \liminf_{t\to \infty} (\phi_t(E)-\phi_t(F)) <+\infty.
    \end{equation*}
\end{df}
By \cite[Lemma 2.17]{leistner2023stability}, $\sim^{\text{inf}}$ is an equivalence relation on the set $\slicing$ of limit semistable objects.
We will denote by $\slicing /\sim^{\text{inf}}$ the set of equivalence classes, it is a finite set by  \cite[Example 2.42 and Lemma 2.35]{leistner2023stability}.

The asymptotics of the phases define a total order relation on the set $\slicing/\sim^{\text{inf}}$.
\begin{df}[{\cite[Definition 2.16]{leistner2023stability}{}}]
    Given two limit semistable objects $E,F$ we will write $E<^{\text{inf}} F$ if  
    \begin{equation*}
        \liminf_{t\to \infty} (\phi_t(F)-\phi_t(E))=+\infty.
    \end{equation*}
\end{df}
The relation $<^{\text{inf}}$ is a total order relation on $\slicing/\sim^{\text{inf}}$ by \cite[Lemma 2.17]{leistner2023stability}.

For $E\in \slicing$ we will denote by $\D^E$ the full subcategory of objects whose limit semistable HN factors are $\sim^{\text{inf}}$ equivalent to $E$.
Clearly if for $E,F\in \slicing$ we have  $E\sim^{\text{inf}} F$ then $\D^E=\D^F$.
The set $\{D^E|E\in \slicing/\sim^{\text{inf}}\}$ is finite and has a total order induced by $<^{\text{inf}}$.
\begin{lem}[{\cite[Proposition 2.20]{leistner2023stability}} ]\label{Proposition_2.20_HLJR}
For any $E\in \slicing$ the categories $\D^E$ are thick triangulated subcategories that gives an SOD
\begin{equation*}
    \D=\langle\D^E|E\in \slicing/\sim^{\text{inf}}\rangle
\end{equation*}
where the order of the components of the SOD is the one induced by $<^{\text{inf}}$.
\end{lem}

\section{Algebraic stability conditions}
In this section we review a construction of stability conditions due to Macr{\`i}, see \cite{Macri07StabCurves} for the details.
\subsection{Exceptional objects}

Consider the bounded derived category $\D$ of coherent sheaves of a smooth projective variety.
Let us define $\Hom^k(A,B):=\Hom(A,B[k])$ for objects $A,B\in \D$.
We also define $\Hom^{\bullet}(A,B):=\bigoplus_{k\in \Z}\Hom^k(A,B)$.
\begin{df}[Exceptional objects]
Let $E,E_1,\dots,E_n\in \D$ 
\begin{itemize}
    \item $E$ is called \emph{exceptional} if $\Hom^k(E,E)=0$ for $k\neq 0$ and $\Hom(E,E)=\C$,
    \item the ordered set 
    $\{E_1,\dots,E_n\}$ is called an \emph{exceptional collection} if all the objects are exceptional and $\Hom^{\bullet}(E_i,E_j)=0$ for $j<i$.
\end{itemize}
\end{df}

\begin{df}[\cite{Macri07StabCurves}]
    Let $\mathcal{E}:=\{E_1,\dots, E_n\}$ be an exceptional collection, we call $\mathcal{E}$
    \begin{itemize}
        \item \emph{strong} if $\Hom^k(E_i,E_j)=0$ for all $i,j$ and $k\neq 0$,
        \item \emph{full} (or \emph{complete}) if it generates $\D$ by shifts and extensions, 
        \item $\mathrm{Ext}$ if $\Hom^{\leq 0}(E_i,E_j)=0$ for all $i\neq j$.
    \end{itemize}
\end{df}

Two classical examples are due to Beilinson: the derived category of the projective space $\D^b(\pr^{N-1})$ has two standard full strong exceptional collections $\{\o,\o(1),\dots, \o(N-1)\}$ and $\{\Omega^{N-1}(N-1), \dots, \Omega, \o\}$.
Using the first exceptional collection we get that the category $\D^b(\pr^{N-1})$ is equivalent to the derived category of the quiver $T^{N-1}$ with $N$ vertices
$$X_0\map\cdots\map X_{N-1} $$
the arrows are $\phi_i^j:X_i \map X_{i+1}$ for $i=0, \dots , N-2$ and $j=0, \dots, N-1$ and relations $\phi^j_{i+1}\circ \phi^k_i=\phi^k_{i+1}\circ \phi^j_i$.

\begin{example}[Exceptional collections on $\pr^2$]\label{example_exc_coll_P2}
    In the case of $\pr^2$, by \cite{Gorodentsev1987ExceptionalVB,Drezet1985FibrsSE}
    all full exceptional collection of coherent sheaves are strong and equivalent, up to mutation and shift, to the Beilinson exceptional collection.
    Pirozkhov classified the admissible subcategories of $\mathrm{D^b}(\pr^2)$:  it is proved in \cite{pirozhkov2020admissible} that a full triangulated subcategory $\A\subseteq \mathrm{D^b}(\pr^2)$ whose inclusion have left and right adjoint is generated by a subset of 
    $\{E_0,E_1,E_2\}=\mathcal{E}$, where $\mathcal{E}$ is an exceptional collection obtained by mutation from $\{\o,\o(1),\o(2)\}$.
\end{example}

\subsection{Stability conditions from exceptional collections}
In this section we recall and refine a standard construction of stability conditions due to Macr{\`i}, see \cite{Macri07StabCurves}.
Such stability conditions can be constructed on any triangulated category with a full strong exceptional collection, they are called \emph{algebraic}.
Our improvement consists in giving a necessary and sufficient condition on the phases of the exceptional objects.
For a real number $x\in\R$ we denote by $\ceil{x}$ the smallest integer that is bigger or equal to $x$.

\begin{thm}[{\cite[Section 3]{Macri07StabCurves}}]\label{thm_stabcond_fromexc}
    Let $\mathcal{E}=\{E_0,\dots E_{n}\}$ be a full strong exceptional collection of $\D$. Then for any $m_i\in \R_{>0}, \phi_i\in \R,i=0, \dots,n$ such that $\ceil{\phi_i}<\phi_{i+1}, i=0, \dots n-1$ there exists a unique stability condition $\sigma=(Z,\slicing)$ such that $E_i$ are stable of phase $\phi_i$ and $Z(E_i)=m_ie^{\ir \pi \phi_i}$.
\end{thm}
\begin{proof}
    We observe that between subsequent $\phi_i$ there is always an integer,
    therefore the unique choice of integers $p_0,\dots,p_n\in \Z$ such that 
    $\phi_i+p_i\in (0,1]$ satisfies
    $p_0>p_1>\cdots>p_n$.
    It is also straightforward to check that 
    $\{E_0[p_0],\dots,E_{n}[p_n]\}$ is an $\Ext$ exceptional collection.

    We now follow \cite{Macri07StabCurves}.
    By induction on $n$ we can prove that  $\mathcal{Q}:=\langle  E_0[p_0],\dots,E_{n}[p_n] \rangle$ is a heart of a bounded $t-$structure, see \cite[Lemma 3.14]{Macri07StabCurves}, in particular it is abelian.
    We can prove (again by induction on $n$ and diagram chasing) that $\mathcal{Q}$ is of finite length  (i.e. artinian and noetherian).
    We define the central charge $Z : K(\mathcal{Q})\to \C$ by $Z(E_i[p_i]):=m_ie^{\pi\ir (\phi_i+p_i)}$ for $i=0,\dots,n$.
    Then by \cite[Propositions 2.4 and 5.3]{2007_stab_cond_on_triang_cat} we get a stability condition with central charge $Z$.
    Moreover $E_i[p_i]$ are the only simple objects in $\mathcal{Q}$: for any $i\in\{0,\dots,n\}$ all monomorphisms $Q\to E_i[p_i],$ $Q\in \mathcal{Q}$  are indeed isomorphisms.
    The last remark implies that the $E_i[p_i]$ are stable objects.
    The uniqueness in the claim follows form  \cite[Proposition 5.3]{2007_stab_cond_on_triang_cat}.
    
   Let us check the support property.
   We start by noticing that
   \begin{equation*}
           \sup \left\{\frac{\norm{\ch(E)}}{|Z(E)|}|\text{ for } E \text{ semistable} \right\}=\sup\left\{\frac{\norm{\ch(Q)}}{|Z(Q)|}|\text{ for } Q\in \mathcal{Q} \text{ semistable}\right\}
   \end{equation*}
   For any $Q\in \mathcal{Q}$ we have $\ch(Q)=\sum_{i=0}^n \ch (E_i[p_i]^{\oplus b_i})$ for $b_i\in \Z_{\geq 0}, i=0,\dots,n$ and 
   $Z(Q)=\sum_{j=0}^n b_j Z(E_j[p_j])$
   let us denote by $C':=\max{\norm{\ch(E_i)}}$.
   We observe that
   \begin{equation*}
       \norm{\ch(Q)}\leq C'\sum b_i
   \end{equation*}
    moreover 
    \begin{equation*}
        \frac{|Z(Q)|}{\sum b_i}= |\sum_j \frac{b_j}{\sum_i b_i}Z(E_j[p_j])|.
    \end{equation*}
    We note that the sum of $\alpha_j:=\frac{b_j}{\sum_i b_i},j=0,\dots n$ is one,
    so $\sum_j \alpha_j Z(E_j[p_j])$ is in the convex hull $\mathrm{Conv}(Z(E_j[p_j]))$ of $Z(E_j[p_j]),j=0,\dots,n$.
    The set $\mathrm{Conv}(Z(E_j[p_j]))$ is a compact polygon in the upper half plane, the distance between the polygon and the origin is $C'':=\min{m_i}>0$.
    The last remark shows that 
    \begin{equation*}
        \frac{|Z(Q)|}{\sum b_i}\geq C''.
    \end{equation*}
    We conclude that 
    \begin{equation*}
        \frac{\norm{\ch(Q)}}{|Z(Q)|}\leq \frac{C'}{|Z(Q)|}\sum b_i\leq \frac{C'}{C''}.
    \end{equation*}
\end{proof}

\begin{rmk}\label{rmk_alg_stab}
    Let us note that in the construction above the constant of the support property $\frac{C'}{C''}$ depends only on the masses of the excpetional objects and not on the phases.
\end{rmk}
Given a strong exceptional collection $\mathcal{E}=\{E_0,\dots,E_{n}\}$, following \cite{Macri07StabCurves}, we define the set 
$\Theta_{\mathcal{E}}$ of stability conditions arising from Theorem~\ref{thm_stabcond_fromexc} up to the $\Tilde{\GL}^+_2(\R)-$action.
By the deformation theorem it is quite clear that $\Theta_{\mathcal{E}}\subset \Stab(\D)$ is an open subset.
\begin{df}[Pure algebraic stability conditions]
    A stability condition $\sigma$ is called \emph{pure} relative to a full exceptional collection  
    $\{E_0,\dots,E_{n}\}$ if the only stable objects are $E_0,\dots,E_{n}$ and their shifts.
\end{df}
\begin{prop}[Pure algebraic stability conditions]\label{prop_pure_alg_stabcond}
    Let us consider $\sigma\in \Theta_{\mathcal{E}}$ where $\mathcal{E}=\{E_0,\dots,E_{n}\}$ is a strong exceptional collection and let us denote by $\phi_i:=\phi(E_i)$ the phases of the exceptional objects.
    If $\phi_{i+1}-\phi_i>2$ then the only stable objects are $E_j$ and their shifts. 
\end{prop}
\begin{proof}
    Up to the $\Tilde{\GL}^+_2(\R)-$action, $\sigma=(Z,\slicing)$ is given on the heart $\slicing((0,1])=\langle  E_0[p_0],\dots,E_{n}[p_n] \rangle$ with $p_0>p_1>\cdots>p_n$.

    To prove the claim it is enough to show that each object in the heart is the direct sum of the generators. 
    So it is enough to show that $\Ext^1(E_i[p_i],E_j[p_j])=0$ for $i<j$.
    By assumption $\phi_{i+1}-\phi_i\geq 2$ so $p_{i+1}\leq p_i-2$, see the construction of $\sigma$ in the previous theorem.
    We notice that  $1-p_i+p_j\leq 1-p_{j-1}+p_j\leq -1$ so $\Ext^1(E_i[p_i],E_j[p_j])=\Ext^{1-p_i+p_j}(E_i,E_j)$ is zero since the exceptional collection $\mathcal{E}$ is strong.
\end{proof}
\begin{rmk}
    When $\D=\mathrm{D^b}(\pr^2)$ we have a stronger version of the lemma above: the only assumption we need in this case is $\phi_{i+1}-\phi_i>1$.
    See \cite[Lemma 2.4]{2017_li_stab_P2} for a proof of this particular case.
\end{rmk}

\subsection{Geometric stability conditions}
There are two standard ways to construct stability conditions: one is via slope and tilt (and iterations), the other is via exceptional collection as in the previous paragraph (algebraic stability conditions).
The first type of stability is usually called geometric stability condition but we will use this terminology with a different meaning, we will call such stability conditions slope-kind stability conditions.
In higher dimensions slope-kind stability conditions were constructed in dimension three on abelian and  Fano varieties, quintic threefolds and the projective space, see \cite{2019_li_stab_quintic_3fold,Bernardara_2017,2016stab_cond_threefolds,Macr_2014_stabP3}.
\begin{df}[Geometric stability condition]\label{def_gem_stabcond}
    Let $\D$ be the bounded derived category of coherent sheaves on an algebraic variety. 
    A stability condition $\sigma\in \Stab(\D)$ is called \emph{geometric} if all skyscraper sheaves of closed points are stable of the same phase.
\end{df}

The slope-kind stability conditions are geometric, see \cite{2019_li_stab_quintic_3fold,Bernardara_2017,2016stab_cond_threefolds,Macr_2014_stabP3}. The goal of this paragraph is to prove that some stability conditions constructed as in Theorem~\ref{thm_stabcond_fromexc} are geometric, see Propositions~\ref{prop_beilinsonexc_geom_stab} and~\ref{prop_geom_alg_stabcond_P2}.

\begin{prop}\label{prop_stab_fasi_vicine_generale}
    Let $\mathcal{Q}:=\langle E_0,\dots,E_n\rangle$ be an abelian category generated by simple non zero objects, $Z:K(\mathcal{Q})\to \C$ a stability function 
    {\cite[Definition 2.1]{2007_stab_cond_on_triang_cat}{}}.
    Let us consider $Q\in\mathcal{Q}$ with
    \begin{equation*} 
    \xymatrix{
    0=A_0 \ar@{^{(}->}[r]& A_1 \ar@{^{(}->}[r]\ar@{->>}^{\psi_1}[d]&A_2 \ar@{^{(}->}[r]\ar@{->>}^{\psi_2}[d] &\cdots \ar@{^{(}->}[r]& A_{m-1}\ar@{^{(}->}[r]&A_m=Q\ar@{->>}^{\psi_m}[d]\\
               &E_{i_1}          &E_{i_2} &         &  & E_{i_m} 
    }
    \end{equation*}
    where any extension is not trivial and $\phi(E_{i_j})<\phi(E_{i_{j+1}})$. Then $Q$ is stable and $A_{i}$ for $1=1,\dots,m-1$ are the only subobjects of $Q$.
\end{prop}
\begin{proof}
Let us start by proving the second claim.
    Let us consider a non zero subobject $Q'\subset Q$.
    The image $\psi_m(Q')$ can be $0$ or $E_{i_m}$, in both cases $Q'\cap A_{m-1}\neq 0$, indeed in the second case a splitting would occur and we have assumed that the extensions are non trivial.
    Arguing in the same way for $\psi_{m-1},\dots,\psi_1$ we get that $E_{i_1}\subseteq Q'$.
    If $A_1=Q'$ we are fine otherwise we consider the non zero subobject $Q'/A_1\subseteq Q/A_1 $ and prove similarly that 
    $A_2/A_1\subseteq Q'/A_1$ so $A_2\subseteq Q'$. By iteration we get the second claim.
    
    To prove that $Q$ is stable we will prove by induction that  $\phi(A_j)<\phi(A_{j+1})<\phi(E_{i_{j+1}})$ for $j=1,\dots,m-1$.
    For $j=1$ it follows by construction that $\phi(A_1)<\phi(A_2)<\phi(E_{i_2})$.
    By induction hypothesis we assume the claim for $j=1,\dots,m-2$.
    In particular we know that $\phi(A_{m-1})<\phi(E_{i_{m-1}})<\phi(E_{i_m})$ and so $\phi(A_{m-1})<\phi(A_m)<\phi(E_{i_m})$.
\end{proof}
We now apply the previous result to a specific example.
\begin{prop}\label{prop_beilinsonexc_geom_stab}
    Let us consider the Beilinson exceptional collection $\mathcal{E}:=\{\Omega^{j}(j)\}$  of $\pr:=\pr^{N-1}$ and a stability condition $\sigma\in \Theta_{\mathcal{E}}$ as defined in Theorem~\ref{thm_stabcond_fromexc} with $\phi_{i+1}-\phi_{i}<1$ for $i=0,\dots,N-1$.
    Then every skyscraper sheaves of points $\o_{\pt}$ is stable i.e.,  $\sigma$ is geometric.
\end{prop}
\begin{proof}
    Let us recall that $\sigma$ is constructed (up to $\Tilde{\GL}^+_2(\R)$ action) on the abelian category
    \begin{equation*}
        \mathcal{Q}:=\langle \Omega^{N-1}(N-1)[p_0+N-1],, \dots,\Omega(1)[p_0+1], \o[p_0]\rangle
    \end{equation*}
    where the object $\Omega^i(i)[p_0-i]$ has phase $\phi_{N-1-i}+p_0+i$.
    We have a resolution
    \begin{equation*}
        0\to \Omega^{N-1}(N-1)\to\cdots\to \o\to \o_{\pt}\to 0
    \end{equation*}
    and its stupid filtration gives us the following exact triangles
    \begin{equation*}
    \xymatrix{
    0 \ar[r]& \bullet \ar[r]\ar[d]&\bullet \ar[r]\ar[d] &\cdots \ar[r]& \bullet\ar[r]&\o_{\pt}[p_0]\ar[d]\\
               &\o[p_0] \ar@{-->}[ul]         &\Omega^1(1)[p_0+1]\ar@{-->}[ul] &         &  & \Omega_{N-1}(N-1)[p_0+N-1]\ar@{-->}[ul] 
    }
\end{equation*}
We observe that $\phi(\Omega^{i+1}(i+1)[p_0+i+1])-\phi(\Omega^i(i)[p_0+i])>0$ so we can apply Proposition~\ref{prop_stab_fasi_vicine_generale} and get the claim.
\end{proof}

For the projective plane we have a stronger result.
\begin{prop}[Geometric algebraic stability conditions for the projective plane {\cite[Proposition 2.5]{2017_li_stab_P2}}]\label{prop_geom_alg_stabcond_P2}
    Consider the derived category of the projective plane $\D=\mathrm{D^b}(\pr^2)$, any strong full exceptional collection of coherent sheaves $\mathcal{E}$ and   $\sigma\in \Theta_{\mathcal{E}}$ with $\phi_{i+1}-\phi_i<1$ then all skyscraper sheaves of points are stable of the same phase i.e.,  $\sigma$ is geometric.
\end{prop}

\section{Quantum cohomology, Gamma Conjecture II, and Quantum cohomology central charge}

In the Section~\ref{subsection_quan_coh} we recall some general definitions of  quantum cohomology and quantum connection for a fixed Fano variety $F$, see \cite{dubrovin1994geometry} for general definitions.
In the Section~\ref{subsection_gammaconj_quantZ} we state the   \cite[Gamma Conjecture II]{Galkin_2016} and define the quantum cohomology central charge as in \cite{Iritani_2009,Galkin_2016};
for the details see
\cite{dubrovin1994geometry,dubrovin1998painleve,Galkin_2016,cotti2018helix}.
In the Section~\ref{subsection_quant_of_proj_spaces} we review the example of small quantum cohomology of projective spaces.

\subsection{Quantum cohomology, quantum connection and flat sections}\label{subsection_quan_coh}
In this section we will follow \cite{manin_frob_man, Galkin_2016}.
We consider a smooth projective variety  $F$ with ample anticanonical bundle, its cohomology with complex coefficients $\H^{\bullet}(F,\C)$ and a homogeneous basis 
$\{\phi_i\}$. We will use the notations $\H=\H^{\bullet}(F)=\H^{\bullet}(F,\C)$.
We will write an element $\tau\in \H$ as $\tau=\sum \tau^{i}\phi_i$ where $\tau^{i_2}, \dots, \tau^{i_2+b_2-1}$ are the coordinates of $\H^2(F,\C)$ and $b_i, i=0,\dots, 2\dim (F)$ are the Betti numbers of $F$.

The \emph{genus zero Gromov--Witten potential} is a formal power series 
\begin{equation*}
\mathcal{F}^F_0(\tau)\in \C[[\tau^0, \dots,\tau^{i_1+b_1-1}, e^{\tau^{i_2}}, \dots, e^{\tau^{i_2+b_2-1}}, \tau^{i_3}\dots, \tau^{i_{2\dim(F)}+b_{2\dim(F)}-1}]]
\end{equation*}
defined by the Gromov--Witten invariants.
Since $F$ is Fano we have that for $\tau\in H^2(F,\C)$ the power series $\mathcal{F}^F_0(\tau)$ is a finite sum.
We will denote by $B\subseteq \H$ the open domain of convergence of $\mathcal{F}^F_0$. We will assume that $B$ is nonempty.
We define the quantum product at $\tau\in B$ as follows: for $\alpha, \beta,\gamma\in \H$
\begin{equation*}
    (\phi_i\star_{\tau}\phi_j ,\phi_k)_F=\partial_i\partial_j\partial_k \mathcal{F}^F_0(\tau)
\end{equation*}
where $(\cdot,\cdot)_F$ is the Poincaré pairing on $F$.
The product $\star_{\tau}$ defines a Frobenius manifold structure on $B$, see \cite[Definition 2.1]{dubrovin1998painleve}.
When we consider $\tau \in \H^2(F,\C)$ we call $\star_{\tau}$ the \emph{small quantum product}.

We consider the trivial vector bundle $\H\times B\times \pr^1\to B\times\pr^1$.
Fixing a coordinate $[z;1]\in \pr^1$ and denoting by $\alpha\in \H\cong \mathrm{T}_{B,\tau}$ a tangent vector, we define the quantum connection $\Tilde{\nabla}$ as follows
\begin{equation*}\label{eq_big_quant_conn}
    \begin{split}
    \Tilde{\nabla}: \H\oplus \mathrm{T}_{\pr^1,z}& \to (\H\oplus \mathrm{T}_{\pr^1,z})\otimes\Omega_{B\times\pr^1}\\
        \Tilde{\nabla}_{\alpha}&=\partial_{\alpha}+\frac{1}{z}(\alpha\star_{\tau})\\
        \Tilde{\nabla}_{z\partial_z}&=z\partial_z-\frac{1}{z}(\mathcal{E}\star_{\tau})+\mu
    \end{split}
\end{equation*}
where $\mu$ is the grading operator defined by 
\begin{equation*}
    \mu|_{\H^{2p}(F)}=(p-\frac{\mathrm{dim}(F)}{2})\mathrm{id}_{\H^{2p}(F)},
\end{equation*}
\begin{equation*}
    \mathcal{E}:=\mathrm{c}_1(F)+\sum(1-\frac{1}{2}\deg(\phi_i))\tau^i\phi_i 
\end{equation*}
is the 
\emph{Euler vector field}.
The connection $\Tilde{\nabla}$ is  meromorphic and flat.

Restricting to $\tau=0$ we have the following connection 
\begin{equation*}\label{eq_quantconn_GGI2.2.1}
    \nabla_{z\partial_z}=z\partial_z-\frac{1}{z}(c_1(F)\star_0))+\mu
\end{equation*}
on $\H\times \pr^1\to \pr^1$.
There is a \emph{canonical} fundamental solution of flat sections of $\nabla$.
\begin{prop}[{\cite[Proposition 2.3.1]{Galkin_2016}{}}]\label{prop_GGI2.3.1}
    There exists a unique holomorphic function 
    \[S:\pr^1\setminus \{0\}\to \mathrm{End}(H^{\bullet}(F))\]
    with $S(\infty)=\mathrm{id}$ such that 
    \begin{equation}\label{eq_fund_sol}
        \nabla (S(z)z^{-\mu}z^{\rho}\alpha)=0 \text{ for all } \alpha\in \H^{\bullet}(F)
    \end{equation}
    where $\rho=(c_1(F)\cup)\in \mathrm{End}(\H^{\bullet}(F))$ and we define $z^{-\mu}=\exp{(-\mu \log(z))}, z^{\rho}=\exp{(\rho \log(z))}$.
\end{prop}
\begin{rmk}[{\cite[Remark 2.3.2]{Galkin_2016}{}}]\label{rmk_can_sol_fortau_non_0}
    There is also  a canonical fundamental solution of $\Tilde{\nabla}=0$ on $B\times(\pr^1\setminus\{0\})$ that extends the solution in~\eqref{eq_fund_sol}, more precisely there is
    a holomorphic function
    \begin{equation*}
        S(\tau,z):B\times (\pr^1\setminus\{0\})\to \End(\H^{\bullet}(F))
    \end{equation*}
    such that 
    \begin{equation*}\label{eq_fund_sol_gen}
        \Tilde{\nabla}(S(\tau, z)z^{-\mu}z^{\rho}\alpha)=0 \text{ for all } \alpha\in \H^{\bullet}(\pr)
    \end{equation*}
    satisfying $S(\tau, \infty)=\mathrm{id}$.
\end{rmk}

\begin{example}
    When $F$ is a  Grassmannian the small quantum product $\star_{\tau}$ is \emph{semisimple} for all $\tau\in \H^2(F)$ i.e.,  the ring $(\H, \star_{\tau})$ is isomorphic to some product of $\C$.
    Note also that  semisimplicity of $(\H,\star_{\tau})$ is equivalent to the existence of a basis $\{e_i\}$ of $\H$ such that $e_i\star_{\tau}e_j=\delta_{ij}e_i$.
    Semisimplicity is an open condition, see \cite{cox_katz} for details.
    See also \cite{manin_frob_man} for details about the non emptiness of the domain of convergence of the Gromov--Witten potential for the projective spaces.
\end{example}

For the rest of this section let us fix $\tau_0\in B$ such that the product $\star_{\tau_0}$ is semisimple.
In this case the matrix $\mathcal{E}\star_{\tau_0}$ is diagonalizable; let us denote by $\{u_j(\tau_0)\}$ its eigenvalues.
We say that $\phi\in \R$ is an \emph{admissible phase} for $\{u_j(\tau_0)\}$ if 
\begin{equation*}
    e^{-\ir \phi}(u_j(\tau_0)-u_i(\tau_0))\notin \R_{>0}
\end{equation*}
for any non zero difference $u_j(\tau_0)-u_i(\tau_0)$.

Let us denote by $\Psi$ the matrix whose columns $\Psi_j$ form a basis of idempotents of $\star_{\tau_0}$; we can assume that they are normalized with respect to the Poincaré pairing on $F$.
Let us observe that $\Psi_j$ are
eigenvectors of $\mathcal{E}\star_{\tau_0}$. 
We will denote by
$U=\mathrm{diag}(\dots,u_j(\tau_0),\dots)$ the diagonal matrix of eigenvalues 
associated to the columns of $\Psi$ i.e., $U$ is the linear operator $\mathcal{E}\star_{\tau_0}$ in the basis $\{\Psi_j\}$.
\begin{prop}[{\cite[Proposition 2.5.1]{Galkin_2016}}]\label{prop2.5.1_GGI}
    Choose a $\tau_0\in B\subseteq\H^{\bullet}(F,\C)$ such that $\star_{\tau_0}$ is semisimple.
    Consider the quantum connection
    $\Tilde{\nabla}$ on $B\times\pr^1$ and let $\phi\in\R$ be an admissible phase for $\{u_j(\tau)\}$.
    Then in a neighborhood of $\tau_0$ we have an analytic fundamental solution $Y_{\tau}(z)$ of $\Tilde{\nabla}$ and $\epsilon>0$ such that
    \begin{equation*}
        Y_{\tau}(z)e^{U/z}\to \Psi \text{ as } z \to 0 \text{ in the sector } |\mathrm{arg}(z)-\phi|<\frac{\pi}{2}+\epsilon
    \end{equation*}
    see Figure~\ref{fig_domain_asym_exp_fund_sol}.
    The solution $Y_{\tau}(z)$ is called the \emph{asymptotically exponential fundamental solution} and it is the unique solution satisfying this asymptotic.
\end{prop}
\begin{figure}[ht]
\includegraphics[width=8cm]{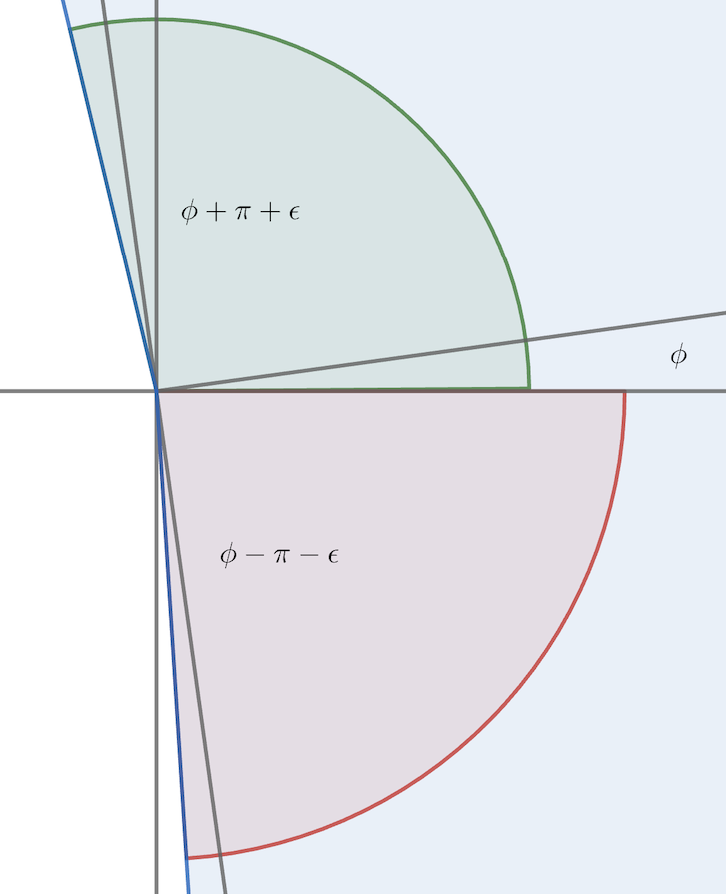}
\caption{Domain where we estimate the asymptotic behaviour of the of the asymptotically exponential fundamental solution. }
\label{fig_domain_asym_exp_fund_sol}
\end{figure}
\subsection{Gamma Conjecture II and quantum cohomology central charges}\label{subsection_gammaconj_quantZ}
In 1998  Dubrovin conjectures that for a Fano manifold the semisiplicity of the big quantum cohomology is equivalent to the existence of a full exceptional collection in the derived category, see \cite{dubrovin1998geometryICM}. 
The Gamma Conjecture II is a refinement of Dubrovin's conjecture.
In this section we will state the latter and review its proof for the projective spaces.
Let us observe that if $\star_{\tau_0}$ is semisimple then Proposition~\ref{prop2.5.1_GGI} gives us a basis of flat sections $y_j(z,\tau)$ i.e., $y_j(z,\tau)$ are the columns of $Y_{\tau}(z)$ for $\tau$ near $\tau_0$.
By $\Tilde{\nabla}-$parallel transport we get a basis of solutions $y_j(z)$ of $\nabla=\Tilde{\nabla}_{|_{\tau=0}}$ on $\{0\}\times \R_{>0}\subseteq B\times \pr^1$.

On the other hand Proposition~\ref{prop_GGI2.3.1} gives a canonical isomorphism
\begin{equation}\label{eq_iso_Phi}
\begin{split}
    \Phi : \H^{\bullet}(F)&\to \left\{ s:\R_{> 0}\to \H^{\bullet}(F): \nabla s=0   \right\}\\
    \alpha &\mapsto (2\pi )^{-\frac{\mathrm{dim}(F)}{2}}S(z)z^{-\mu}z^{\rho}\alpha
\end{split}
\end{equation}
where we use the standard determination $\ln(z)\in \R$, for $z\in \R_{>0}$, in the expression
\[
z^{-\mu}z^{\rho}=\exp{(-\mu\ln{z})}\exp{(\rho \ln{z})}.
\]
We will use the following notation for the modified Chern character 
\begin{equation*}
    \Ch:=(2\pi\ir)^{\deg}\ch.
\end{equation*}
We recall that the Gamma class $\hat{\Gamma}$ of $F$ is
\begin{equation*}
    \hat{\Gamma}:=\prod_{i=0}^{\dim (F)}\Gamma(1+\delta_i) \in \H^{\bullet}(F,\R)
\end{equation*}
where $\delta_i$ are the Chern roots of the tangent bundle of $F$ and $\Gamma$ is the Euler's Gamma function.
\begin{conj}[{\cite[Gamma Conjecture II]{Galkin_2016}{}}]\label{conj_gammaII}
    Let $F$ be a Fano variety with $\tau_0\in B\subseteq \H^{\bullet}(F,\C)$ such that $\star_{\tau_0}$ is semisimple and $\mathrm{D^b}(F)$ admits a full exceptional collection.
    Then there exists a full exceptional collection $\{E_j\}$ such that
    $\Phi(\hat{\Gamma}\cdot\Ch(E_j))=y_j(z)$.
\end{conj}
\begin{thm}[\cite{Galkin_2016,cotti2018helix}]
    The Gamma Conjecture II holds for Grassmannians.
\end{thm}

In \cite{Iritani_2009,Galkin_2016} Iritani and Galikin--Golyshev--Iritani  propose a \text{quantum cohomology central charge} defined as follows.
For a vector bundle $V$ on $F$, we set
\begin{equation}\label{eq_quantum_central_charge}
    Z(V):=(2\pi z)^{\mathrm{dim}(F)/2}\int_{F}\Phi(\hat{\Gamma}\cdot\Ch(V))(z).
\end{equation}
Note that $Z(V)$ is defined for $z\in\R_{>0}$ but it can be analytically continued on the universal cover of $\C^*$.

\begin{rmk}\label{rmk_gammaII_implies_asympZ}
    Suppose that $F$ satisfies the Gamma Conjecture II and that quantum product  $\star_0$ at $\tau_0=0$ is semisimple.
As above we consider  the idempotent basis $\{\Psi_j\}$ for $\star_{0}$ i.e., $\Psi_j\star_0\Psi_i=\delta_{ij}\Psi_j$.
We observe that  $\{\Psi_j\}$ is a basis of eigenvectors of $\mathcal{E}\star_{0}$ and  $\mathcal{E}_{\tau=0}=\mathrm{c}_1(F)$.
Let us denote by $u_j=u_j(0)$ the eigenvalues of $\mathrm{c}_1(F)\star_0$ i.e., $\mathrm{c}_1(F)\star_0\Psi_j=u_j\Psi_j$ and let us fix $\phi\in\R$ an admissible phase.
We denote by $\{E_j\}$ the full exceptional collection corresponding to the asymptotically exponential fundamental solutions of $\nabla$, via $\Phi$.
Then we have the following asymptotics
\begin{equation}\label{eq_asym_quant_central_charge}
    Z(E_j)\sim \sqrt{(\Psi_j,\Psi_j)_F}(2\pi z)^{\mathrm{dim}(F)/2}e^{-u_j/z}
\end{equation}
for $z\to 0^+$ on $|\arg(z)-\phi|<\frac{\pi}{2}+\epsilon$, where $(\cdot,\cdot)_F$ is the Poincaré pairing on $\H^{\bullet}(F,\C)$.
For the details see \cite[Paragraphs 3.8 and 4.7]{Galkin_2016}.
\end{rmk}

\subsection{Quantum cohomology of projective spaces}\label{subsection_quant_of_proj_spaces}

In this final section we recall the explicit description of the quantum cohomology of the projective spaces and we explain the solution of the Gamma Conjecture II for projective spaces.
Let us denote the projective space of dimension $N-1$ by $\pr=\pr^{N-1},N\geq 2$.
We will denote by $h\in \H^2(\pr)$ the class of a hyperplane; we will use the basis $1,h,\dots,h^{N-1}$ for the cohomology.
We observe  that the small quantum product at $\tau\equiv \tau h \in\H^2(\pr,\C)\cong\C$ is given by
\begin{equation*}
    h\star_{\tau}h^i=
    \begin{cases}
        h^{i+1} &\text{ if } 0\leq i\leq N-2\\
        e^{\tau}\cdot1 &\text{ if } i=N-1.
    \end{cases}
\end{equation*}

The quantum product by the Euler vector field at $\tau$ is expressed in this basis by the following matrix
\begin{equation}\label{def_E_tau}
\mathcal{E}_{\tau}:=c_1(\pr)\star_{\tau}=N
    \begin{pmatrix}
        0&0&0&\cdots &0&e^{\tau}\\
        1&0&0&\cdots &0&0\\
        0&1&0&\cdots &0&0\\
        &&&\cdots&&\\
        0&0&0&\cdots&1&0
    \end{pmatrix}.
\end{equation}
The eigenvalues of $c_1(\pr)\star_{\tau}$ are $\zeta^i N e^{\tau/N}$ where $\zeta=\exp{2\pi\ir/N}$.

The operator $\mu$ is given by the following diagonal matrix
\begin{equation*}
    \mu=\mathrm{diag}\left(-\frac{N-1}{2},1-\frac{N-1}{2},2-\frac{N-1}{2},\dots,\frac{N-1}{2}\right).
\end{equation*}
So the quantum connection reads
\begin{equation}\label{eq_quan__connection_proj}
    \nabla_{z\partial_z}=z\partial_z-\frac{1}{z}\mathcal{E}_{\tau=0}+\mu.
\end{equation}
Let us recall the isomorphism~\eqref{eq_iso_Phi}
\begin{equation*}
    \Phi : \H^{\bullet}(\pr)\to \left\{ s:\R_{> 0}\to \H^{\bullet}(\pr): \nabla s=0   \right\},\qquad
    \alpha \mapsto (2\pi )^{-\frac{\mathrm{dim}(\pr)}{2}}S(z)z^{-\mu}z^{\rho}\alpha.
\end{equation*}
Consider the flat sections $y_j(z):=\Phi(\Hat{\Gamma}_{\pr}\Ch(\o(j)))$: for $j=0,\dots,N-1$ they form a basis of the space of solutions of the equation $\nabla  =0$.

\begin{prop}[{\cite[Proposition 3.4.8]{Galkin_2016}}]\label{prop_GGI3.4.8}
    In the notation above we have:
    \begin{enumerate}
        \item [(1)] the limit $v:=\lim_{z\to 0^{+}}e^{N/z}y_0(z)\in \H^{\bullet}(\pr,\C)$ is in the $N$-eigenspace of $\mathcal{E}_{\tau=0}$ and satisfies $(v,v)_{\pr}=1$,
        \item [(2)] let $\Hat{\nabla}_{\partial_{\lambda}}=(\partial_{\lambda}+(\lambda-\mathcal{E}_{\tau=0})^{-1}\mu)$ be the Laplace dual of the quantum connection~\eqref{eq_quan__connection_proj}, then there exists a $\Hat{\nabla}_{\partial_{\lambda}}$-flat section $\varphi(\lambda)$ which is holomorphic near $\lambda=N$ such that 
        $$\varphi(N)=v$$ and 
        \begin{equation*}
            y_0(z)=\frac{1}{z}\int_{N}^{\infty}\varphi(\lambda)e^{-\lambda/z}d\lambda.
        \end{equation*}
    \end{enumerate}
\end{prop}

Proposition~\ref{prop_GGI3.4.8} above gives us the following identities
\begin{equation*}\label{eq_6}
    y_j(z)=\frac{1}{z}\int_{N\zeta^{-j}+\R_{\geq 0}\zeta^{-j}}\varphi_j(\lambda)e^{-\lambda/z}d\lambda,
\end{equation*}
where the expression is valid on $-\frac{\pi}{2}-\frac{2\pi j}{N}<\mathrm{\arg}(z)<\frac{\pi}{2}-\frac{2\pi j}{N}$, with $\zeta=e^{2\pi \ir /N}$ and 
$$\varphi_j(\lambda):=e^{2\pi\ir j\mu /N}\varphi(\zeta^j\lambda).$$

For small $1\gg\phi> 0$ the asymptotically exponential fundamental solutions are given by
\begin{equation*}\label{eq_def_x_j}
x_j(z):=\frac{1}{z}\int_{N\zeta^{-j}+\R_{\geq 0}e^{\ir \phi}}\varphi_j(\lambda)e^{-\lambda/z}d\lambda  
\end{equation*}
in the sector $|\mathrm{arg}(z)-\phi|<\frac{\pi}{2}+\varepsilon$.
More explicitely we have 
\begin{equation*}
    x_j(z)\sim e^{-\zeta^{-j}N/z}\varphi_j(\zeta^{-j}N)\qquad \text{ for } z\to 0, |\mathrm{arg}(z)-\phi|<\frac{\pi}{2}+\varepsilon.
\end{equation*}

In \cite{Galkin_2016} the authors proves that the basis of solutions 
$\{x_j\}$ are associated via $\Phi$ to a mutation $\{E_j\}$ of the exceptional collection 
$\{\o(j)\}$ i.e., $\Phi(\hat{\Gamma}\Ch(E_j))=x_j$: this proves the Gamma Conjecture II for projective spaces.
The mutation is explicit: to get $x_j$ we need to bend the integration path $N\zeta^{-j}+\R_{\geq 0}\zeta^{-j}$ in the expression of $y_j$.
If we choose to bend the path in the clockwise direction, during the journey the radius will meet some $N\zeta^{-(j+1)}, \dots,N\zeta^{-(j+r)}$. 
When the radius meets a $N\zeta^{-i}$ we perform a right mutation.
For example the first mutation takes place when the radius meets $N\zeta^{-(j+1)}$ so we do a right mutation on $\o(j)$ i.e.,
\begin{equation*}
    \{\dots, \o(j),\o(j+1),\dots\}\mapsto \{\dots,\o(j+1), \mathrm{R}_{\o(j+1)}(\o(j)) ,\dots\}
\end{equation*}
see Figure~\ref{figure_bending_mutation}.
Bending the integration paths in anti clockwise direction translates into left mutations, see \cite[Paragraph 2.6]{Galkin_2016}{} for the details.
\begin{figure}[ht]
\includegraphics[width=8cm]{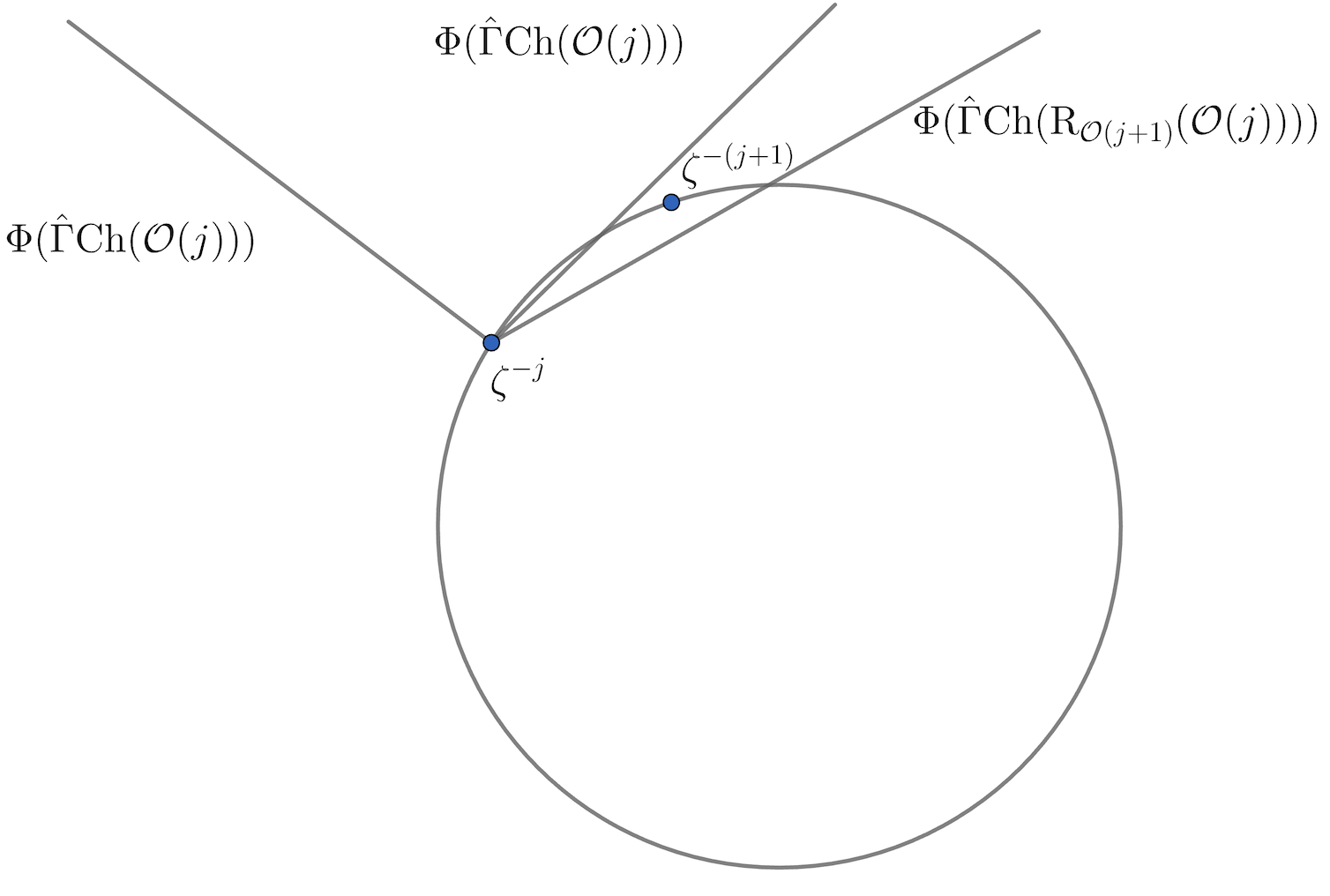}
\caption{Bending the integration path and the associated mutation.}
\label{figure_bending_mutation}
\end{figure}
\begin{rmk}
    Let us note that in the proof of the Gamma Conjecture II for the projective spaces there is no canonical choice of helix foundation for the helix $\{\o(i)\}_{i\in \Z}$. 
    What we mean is that the computations made in \cite[Section 5]{Galkin_2016} holds for all $j\in \Z$.
    In particular we can solve Gamma Conjecture II using the exceptional collection 
    \begin{equation*}
        \left\{\o\left(\ceil*{\frac{N}{2}}\right),\o\left(\ceil*{\frac{N}{2}}+1\right), \dots, \o\left(\ceil*{\frac{N}{2}}+N-1\right)\right\}.
    \end{equation*}
\end{rmk}
\begin{rmk}\label{rmk_asym_quant_Z}
    Let us denote by $\{E_0, \dots E_{N-1}\}$ the full exceptional collection obtained by $\{\o, \dots, \o(N-1)\}$ at the end of the procedure of bending the integration paths that produces $\{x_j\}$ form $\{y_j\}$.
    Then we have that the quantum cohomology central charges 
    $Z(E_j)=(2\pi)^{\frac{N-1}{2}}z^{\frac{N-1}{2}}\int_{\pr}x_{i(j)}$
    where $i:\{0,\dots, N-1\}\to \{0,\dots, N-1\}$ is an explicit  bijection.
    For $\R_{> 0}\ni z=r\to 0$
    we have 
    \begin{equation*}
    \begin{split}
        Z(E_j)&\sim (2\pi r)^{\frac{N-1}{2}}e^{-\zeta^{-i(j)}N/r}\int_{\pr}\varphi_{i(j)}(\zeta^{-i(j)}N)=\\
        &=(2\pi r)^{\frac{N-1}{2}} e^{-\zeta^{-i(j)}N/r}\zeta^{i(j)\frac{N-1}{2}}\int_{\pr}\Hat{\Gamma}_{\pr}.
    \end{split}
    \end{equation*}
\end{rmk}

\begin{rmk}\label{rmk_shifted_beil_exc_collection}
    Remark~\ref{rmk_asym_quant_Z} above applied to the solution of the Gamma Conjecture II associated to 
    \begin{equation}\label{eq_exc_coll_twisted}
        \left\{\o\left(\ceil*{\frac{N}{2}}\right),\o\left(\ceil*{\frac{N}{2}}+1\right), \dots, \o\left(\ceil*{\frac{N}{2}}+N-1\right)\right\}
    \end{equation} 
    reads as follows.

    Let us denote by $\{E_0, \dots E_{N-1}\}$ the full exceptional collection obtained by the exeptional collection in~\eqref{eq_exc_coll_twisted} at the end of the procedure of bending the integration paths that produces $\{x_j\}$ form $\{y_j\}$.
    Then we have that the quantum cohomology central charges 
    $Z(E_j)=(2\pi)^{\frac{N-1}{2}}z^{\frac{N-1}{2}}\int_{\pr}x_{i(j)}$
    with $i:\{0,\dots, N-1\}\to \{\ceil*{\frac{N}{2}},\dots, \ceil*{\frac{N}{2}}+N-1\}$ a bijection.
    For $\R_{> 0}\ni z=r\to 0$
    we have 
    \begin{equation*}
    \begin{split}
        Z(E_j)&\sim (2\pi r)^{\frac{N-1}{2}} e^{-\zeta^{-i(j)}N/r}\zeta^{i(j)\frac{N-1}{2}}\int_{\pr}\Hat{\Gamma}_{\pr}.
    \end{split}
    \end{equation*}
\end{rmk}

\section{Quasi-convergent paths from quantum cohomology}

In~\cite{DHL23noncomm} Halpern-Leistner conjectures many interesting relations between the quantum cohomology of varieties, semiorthogonal decompositions of the derived categories of coherent sheaves, paths in the space of Bridgeland stability conditions and birational geometry.
We state a particular case of {\cite[Proposal III]{DHL23noncomm}}.
\begin{proposal}\label{proposalIII_DHL}
    Let us consider a Fano variety $F$.
    Then there exists a fundamental solution $\Phi'$ of $\Tilde{\nabla}_{|_{\tau}}$ for some $\tau\in B$ and a quasi--convergent path
    $\sigma_r=(Z', \slicing_r)$, $r\in\R_{>0}$ in $\Stab(\mathrm{D^b}(F))$ with
    \begin{equation*}
        Z'(\alpha):=\int_{F}\Phi'(\alpha)
    \end{equation*}
    for $\alpha\in \H^{\bullet}(F,\C)$.
    Furthermore, we have the following \emph{spanning condition:}
    for any $s\in \R$ the real part of an eigenvalue of $\mathcal{E}\star_{\tau}$, the subspace
    \begin{equation*}
        F^s:=\{\alpha\in \H^{\bullet}(F,\C)| \ln{\norm{\Phi'(\alpha)}}\leq \frac{-s}{r} +o(1/r) \text{ as }r\to 0^+\}
    \end{equation*}
    should be spanned over $\C$ by the classes of eventually semistable objects $E\in \mathrm{D^b}(F)$ with $\liminf_{r\to 0^+}\frac{|Z'(E)|}{\norm{\Phi'(E)}}>0$.
\end{proposal}
In  \cite[Remark 13]{DHL23noncomm} the proposed candidate for $Z'$ is the quantum cohomology central charge~\eqref{eq_quantum_central_charge}.
We observe that the easiest way to construct a path of stability conditions  is trying to define it in the algebraic part of the stability space.
Indeed it is enough to  have a path of phases $\phi_i(\bullet)$ and an exceptional collection $\{E_j\}$ satisfying the assumptions in Theorem~\ref{thm_stabcond_fromexc}.
The paths of phases given by the quantum cohomology central charge~\eqref{eq_asym_quant_central_charge} does not always satisfy the assumptions of Theorem~\ref{thm_stabcond_fromexc}, see Corollary~\ref{cor_neg_conj_pr} for a negative result.
However we will fulfill \cite[Remark 13]{DHL23noncomm}, namely we will prove the following conjecture for projective spaces, see Theorem~\ref{thm_conj_pr}.
\begin{conj}\label{conj_sigma_from_qcoh}
    Let us consider a Fano variety $F$, $\tau\in  \H^{2}(F,\C)$, $\Phi^{\tau}$ the canonical fundamental solution of $\Tilde{\nabla}_{|_{\tau}}$ and the induced isomorphism
    \begin{equation*}\label{eq_iso_Phi_tau}
\begin{split}
    \Phi^{\tau} : \H^{\bullet}(F)&\overset{\cong}{\longrightarrow} \left\{ s:\R_{> 0}\to \H^{\bullet}(F): \Tilde{\nabla}_{|_{\tau}} s=0   \right\}\\
    \alpha &\longmapsto (2\pi )^{-\frac{\mathrm{dim}(F)}{2}}S(\tau,z)z^{-\mu}z^{\rho}\alpha.
\end{split}
\end{equation*}
Then there exists a quasi-convergent path
    $\sigma_r=(Z^{\tau}, \slicing_r)$, $z=r\in\R_{>0}$ for $r\to0^+$ with
    \begin{equation}\label{eq_def_Ztau}
        Z^{\tau}(V):=(2\pi r)^{\frac{\mathrm{dim}(F)}{2}}\int_{F}\Phi^{\tau}(\hat{\Gamma}\cdot\Ch(V))
    \end{equation}
    for any vector bundle $V$ on $F$.
\end{conj}
In the case of the plane $\pr^2$ we prove that the SOD induced by a quasi-convergent path as in Conjecture~\ref{conj_sigma_from_qcoh} is $\{\o,\o(1),\o(2))\}$, see Theorem~\ref{thm_conj_P2}.

It is conjectured that there are quasi-convergent paths that start in the geometric part of the stability manifold and end in the algebraic part.
If we had a good description of such paths we could start from a geometric stability condition and classify the semiorthogonal decompositions whose components admit stability conditions.
In \cite[Remark 14]{DHL23noncomm}{} it is asked whether quasi-convergent paths arising in Proposal~\ref{proposalIII_DHL}  start in the geometric part and end in the algebraic part.
In Theorem~\ref{thm_proposalIII_proj} and Corollary~\ref{cor_path_starts_in_geom_proj} we construct 
quasi-convergent paths in the stability manifolds of projective spaces that satisfies such properties, so we answer positively the question in \cite[Remark 14]{DHL23noncomm}{}.

Halpern-Leistner shows that Proposal~\ref{proposalIII_DHL} and the Gamma Conjecture II (see Conjecture~\ref{conj_gammaII}) are essentially equivalent, see \cite[Proposition 25]{DHL23noncomm}.
So in particular Proposal~\ref{proposalIII_DHL} should be true for all Grassmannians, up to some technical verifications.
Our Theorem~\ref{thm_gammaII_proposalIII} is a refined version of \cite[Proposition 25]{DHL23noncomm}{}.

\begin{rmk}
    Note that $Z^{\tau}$ in~\eqref{eq_def_Ztau} is the quantum cohomology central charge $Z$ as defined by Iritani in \cite{Iritani_2009}.
Galkin, Golyshev and Iritani observe that the Gamma Conjectures are  independent of the choice of $\tau\in B$, see \cite[Remark 4.6.3]{Galkin_2016}{} and hence they consider $Z^0=Z$ only, see ~\eqref{eq_quantum_central_charge}.
In our situation the choice of $\tau$ is  important indeed  the asymptotics of the central charges $Z^{\tau}$ depends on $\tau$, see Proposition~\ref{prop_extimated_Ztau}.
\end{rmk}

\subsection{General results}\label{paragraph_general_results}
In this section we start by analyzing the following general situation: we assume that there is a quasi-convergent path and a full exceptional collection of limit semistable objects.
We prove some simple necessary conditions on the limit SOD and on the phases of the exceptional objects, see Proposition~\ref{prop_exc_limitsemist_obj} and Corollary~\ref{cor_increasing}.

In the second part of the section we describe the asymptotic behaviour of the generalized quantum cohomology central charge in terms of the eigenvalues of the Euler vector field, see Proposition~\ref{prop_extimated_Ztau}.
At the end of the section we show how to construct quasi-convergent paths applying Theorem~\ref{thm_stabcond_fromexc}, see  Theorem~\ref{thm_gammaII_proposalIII} and \cite[Proposition 25]{DHL23noncomm}{}.
\subsubsection{SODs given by quasi-convergent paths and exceptional collections}
Let $\D$ be the bounded derived category of coherent sheaves of a smooth projective variety with a full exceptional collection $\langle E_0,\dots, E_d\rangle=\D$.
Consider $\sigma_r=(Z_r,\slicing_r)$, $ r\in (0,\delta)$ a quasi-convergent path (for $r \to 0^+$) that gives a SOD $\langle \D_1,\dots, \D_m\rangle=\D$ in the sense of Proposition~\ref{Proposition_2.20_HLJR}.
We will use the notation $\D_h<\D_{h'}$ to mean $h<h'$.
Let us suppose to know that  $\ln(Z_r(E_j))\sim -u_j/r $ as $r\to 0^+$, for $u_j\in \C^*$.
To compare the order of the components of the SOD and the asymptotics we need to define the lexicographic order on $\{-u_0,\dots,-u_d\}$ in the following way: for $u,v\in \{-u_j\}$
\begin{equation}\label{eq_def_lex_order}
        u\leq v \text{ if }
        \begin{cases}
            \Re(u)<\Re(v) &\text{ or }\\
            \Re(u)=\Re(v) &\text{ and } \Im(u)\leq \Im(v).
        \end{cases}
\end{equation}

\begin{lem}\label{lem_3}
    In the notation above there is an increasing bijection
    \begin{equation*}
        \gimel: \{\D_1, \dots\D_m\}\to \{\Im(-u_j)\}
    \end{equation*}
in particular $m$ is the number of distinct imaginary parts of the complex numbers in $\{-u_j\}$.
\end{lem}

\begin{proof}
    The first step consists in defining $\gimel$.

    For a limit semistable object  $D\in \D_h$ we consider its filtration with respect to the SOD $\langle E_0, \dots,E_d\rangle$, this permits to express
    $Z(D)=\sum a_jZ(E_j),a_j\in \Z$, we define $\gimel'(D)$ to be the biggest $-u_j$ in the lexicographic order~\eqref{eq_def_lex_order} with $a_j\neq 0$.
    Note that the imaginary part of $\gimel'(D)$ gives the asymptotic of the phase of $D$ 
    \begin{equation*}
        \phi_t(D)\sim \frac{\Im(\gimel'(D))}{\pi r}.
    \end{equation*}
    We then define $\gimel(\D_h):=\Im(\gimel'(D))$, the map $\gimel$ is well defined since if $D'\in \D_h$ is an other limit semistable object then $\liminf_{r\to 0^+}\phi_r(D')-\phi_r(D)$ is finite so $\Im(\gimel'(D))=\Im(\gimel'(D'))$.

    If $D\in \D_h$ and $D'\in \D_{h'}$ are limit semistable objects with $h<h'$ then we have $\liminf_{r\to 0^+} \phi(D')-\phi(D)=+\infty$ so $\gimel(D)<\gimel(D')$. This shows that $\gimel$ is strictly increasing, so in particular it is injective.
    Surjectivity of $\gimel$ follows from the following remark:
    the limit semistable filtration of $E_j$ with factors $D_h\in \D_h$ gives 
    \begin{equation*}
        Z(E_j)=\sum_{h=1}^{m} a_hZ(D_h),\qquad a_h\in \{0,1\}.
    \end{equation*}
    So by defining $-u:=\max\{\gimel'(D_h)|a_h\neq 0\}$ we get that $-u=-u_j$
    and $\Im(-u_j)=\gimel(\D_{h'})$ for an index $h'\in \{1,\dots,m\}$ that gives  $-u$.
    We conclude that $\gimel(\D_{h'})=\Im(-u_j)$.
\end{proof}

Given an object $D\in\D$, we denote by $D_{i_{E_j}}\in \D_{i_{E_j}}$ (resp. $D_{f_{E_j}}\in \D_{f_{E_j}}$) the non zero factor with smallest (resp. largest) index in the filtration of $E_j$ given by the SOD $\langle \D_1,\dots, \D_m\rangle=\D$, see Definition~\ref{def_SOD}.
More explicitely we have
\begin{equation*}
    \xymatrix{
    0 \ar[r]& \bullet \ar[r]\ar[d]&\bullet \ar[r]\ar[d] &\cdots \ar[r]& \bullet\ar[r]&E_j\ar[d]\\
               &D_{f_{E_j}} \ar@{-->}[ul]         &D_{\bullet}\ar@{-->}[ul]  &        &  & D_{i_{E_j}}\ar@{-->}[ul] 
    }
\end{equation*}
where all the factors $D_{\bullet}\in \D_{\bullet}$ are assumed to be non zero.
Let us define the map
\begin{equation*}
    \begin{split}
        \F : \{E_0,\dots, E_d\}&\to \{\D_1,\dots,\D_m\}\\
        E_j&\mapsto \D_{f_{E_j}}.
    \end{split}
\end{equation*}

\begin{lem}\label{lem_4}
    If $\Hom^{\bullet}(E_j, E_h)\neq 0$ then $f_{E_h}\geq i_{E_j}$.
\end{lem}
\begin{proof}
    Let us assume that $f_{E_h}< i_{E_j}$ then standard triangulated categories reasoning shows that $\Hom(E_j,E_h)=0$.
    Since the condition $f_{E_h}< i_{E_j}$ is invariant under shifts we get that $\Hom^{\bullet}(E_j, E_h)= 0$.
\end{proof}

\begin{lem}
    In the notation above if $E_j, E_h$  satisfy $\Hom^{\bullet}(E_j, E_h)\neq 0$ and $E_j, E_h$ are both limit semistable then $\Im(-u_j)\leq \Im(-u_h)$.
\end{lem}
\begin{proof}
    Let us suppose $\Im(-u_j)> \Im(-u_h)$, since $E_j,E_h$ are semistable we have that $\Im(-u_i)=\gimel\circ \F(E_i)$ for $i=j,h$.
    Moreover $\gimel$ is strictly increasing so $\F(E_h)<\F(E_j)$ i.e., $f_{E_h}<f_{E_j}$.
    Since $E_j$ is limit semistable $i_{E_j}=f_{E_j}$ and we have that  $f_{E_h}< i_{E_j}$.
    By Lemma~\ref{lem_4} above, we conclude that  $\Hom^{\bullet}(E_j, E_h)=0$, a contradiction.
\end{proof}

\begin{prop}\label{prop_exc_limitsemist_obj}
    In the notation above if we assume that the exceptional objects $\{E_j\}$ are all limit semistable and that $\Hom^{\bullet}(E_j,E_{j+1})=0$, for $j=0,\dots, d-1$, then 
    \begin{equation*}
        \D_h=\langle E_j|\Im(-u_j)=\gimel(\D_h)\rangle
    \end{equation*}
    for $h=1, \dots, m$.
\end{prop}
\begin{proof}
    Since $E_j$ is limit semistable we have $E_j\in \D_h$ for some 
     $h\in\{1, \dots, m\}$; $\gimel$ is a bijection so $h$ is uniquely determined by the condition
    $\gimel(\D_h)=\Im(-u_j)$.
    Thus we have the following inclusion 
    \[
    \langle E_j|\Im(-u_j)=\gimel(\D_h)\rangle\subseteq\D_h.
    \]

    To prove the equalities we observe that if $E_j\in \D_h$ and $E_i\in \D_{h+1}$ then $j<i$, for $h=0, \dots,m-1$.
    Hence the SOD $\langle E_0,\dots, E_d\rangle$ is a refinement of the SOD $\langle \D_0, \dots, \D_m\rangle$. 
    This implies that the filtration of any $E\in \D_h$ given by the SOD $\langle E_0,\dots, E_d\rangle$ have as factors only $E_j\in \D_h$.
\end{proof}

\begin{cor}\label{cor_increasing}
    In the notation above we assume that $\{E_j\}$ are limit semistable and that 
    for $h=1, \dots, m-1$ and $E_j\in\D_h, E_i\in\D_{h+1}$ we have $\Hom^{\bullet}(E_j,E_i)\neq 0$.
    Then the map
    \begin{equation*}
    \begin{split}
        \gimel\circ \F : \{E_0,\dots, E_d\}&\to \{ \Im(-u_i)\}\\
        E_j&\mapsto \Im(-u_j)
    \end{split}
    \end{equation*}
    is increasing.
\end{cor}

\subsubsection{General construction of quasi-convergent paths}
In order to prove Conjecture~\ref{conj_sigma_from_qcoh} for projective spaces we will apply the following two results.
\begin{prop}\label{prop_extimated_Ztau}
    Let $F$ be a Fano variety for which Gamma Conjecture II is true at $\star_{\tau_0}$ and let us denote by $\{E_j\}$ the associated full exceptional collection.  
    Let us denote by $B^0\subseteq \H^{\bullet}(F,\C)$ the connected component containing $\tau_0$ of the region $B$ where $\star$ is defined.
    Fix $\tau\in B^0$ and assume that $\int_{F}\Psi_j^{\tau}\neq 0$, where $\{\Psi_j^{\tau}\}$ is a basis of idempotents of $\star_{\tau}$, see Proposition~\ref{prop2.5.1_GGI}.
    Then  the generalized quantum cohomology central charge has the following asymptotics
    \begin{equation}\label{eq_5}
        \ln{Z^{\tau}(E_j)}\sim \frac{\mathrm{dim}(F)}{2} \ln{(2\pi z)}-u_j(\tau)/z\qquad \text{ as } z\to 0 \text{ in the sector } |\mathrm{arg}(z)-\phi|<\frac{\pi}{2}+\epsilon.
    \end{equation}
\end{prop}
\begin{proof}
    Let us denote by $y_j(\tau,z)$ the $j$-th column of the asymptotically exponential fundamental solution from Proposition~\ref{prop2.5.1_GGI}.
    Since by assumption we have that $\Phi^{\tau_0}(\hat{\Gamma}\Ch(E_j))=y_j(\tau_0,z)$ we get that  the function
    $$\phi_j(\tau,z):=\Phi^{\tau}(\hat{\Gamma}\Ch(E_j))-y_j(\tau,z)$$ 
    is a solution of $\Tilde{\nabla}$ and it vanishes on $(\tau_0,z)$: by unicity of the solution $\phi_j= 0$ is the zero solution on $B^0\times \C^*$.
    So we get that the columns of $Y_{\tau}(z)$ are $\Phi^{\tau}(\hat{\Gamma}\Ch(E_j))$ for every $\tau\in B^0$.
    By Proposition~\ref{prop2.5.1_GGI} we have 
    \begin{equation*}
    e^{u_j(\tau)/z}\int_{F}\Phi^{\tau}(\hat{\Gamma}\Ch(E_j))\to \int_{F}\Psi_j^{\tau}\qquad \text{ as } z\to 0 \text{ in the sector } |\mathrm{arg}(z)-\phi|<\frac{\pi}{2}+\epsilon.
    \end{equation*}
    Since $\int_{F}\Psi_j^{\tau}$ is not zero we take the logarithm on both sides and we get the claim.
\end{proof}
\begin{rmk}
    The assumption $\int_{F}\Psi_j^{\tau}\neq 0$ can be dropped if we are interested in the statement for $\tau$ in a small neighborhood of $0$. 
    Indeed the Gamma Conjecture II at $\tau=0$ with respect to the exceptional collection $\{E_j\}$ gives the asymptotics in~\eqref{eq_asym_quant_central_charge}, so in particular $\int_{F}\Psi_j^{0}\neq 0$ and the same remains true in a neighborhood of $0$.
    We will prove that the property $\int_{F}\Psi_j^{\tau}\neq 0$ is true for $F=\pr$ and any $\tau\in \H^2(\pr,\C)$.
\end{rmk}

The following theorem is a refined version of \cite[Proposition 25]{DHL23noncomm}{}.
\begin{thm}\label{thm_gammaII_proposalIII}
    Let $F$ be a Fano variety 
    and let $\{E_0, \dots,E_d\}$ be any full, strong exceptional collection of $\mathrm{D^b}(F)$. We assume that for a given $\tau\in B$ the asymptotics in~\eqref{eq_5} holds, i.e.,
    \begin{equation*}
        \ln{Z^{\tau}(E_j)}\sim \frac{\mathrm{dim}(F)}{2} \ln{(2\pi z)}-u_j(\tau)/z \qquad\text{ as } z\to 0 \text{ in the sector } |\mathrm{arg}(z)-\phi|<\frac{\pi}{2}+\epsilon.
    \end{equation*}
If the eigenvalues $u_j(\tau)$, $j=0, \dots, d$ have distinct imaginary parts then 
there exists a path $\sigma_r=(Z',\slicing_r)$, $r=z\in (0,\delta)$, for some $\delta>0$, that is quasi-convergent for $r\to 0^+$ and 
    with $Z'(\alpha)=\int_{F}\Phi'(\alpha)$ for $\Phi'$ a fundamental solution of $\Tilde{\nabla}_{|_{\tau}}$ and any $\alpha\in \H^{\bullet}(X,\C)$.
    
    Moreover each $E_j$ is $\sigma_r-$stable for any $r\in (0,\delta)$, the SOD given by $\sigma_{\bullet}$ is $\langle E_0,\dots,E_d\rangle$, and there is a $\mu<0$, that depends only on $\Im(-u_j(\tau))$, $j=0,\dots,d$, such that for $r\sim \delta$  we have $\phi_r(E_{j+1})-\phi_r(E_j)<1+\frac{\mu}{2}$.
\end{thm}
\begin{proof}
    The first step of the proof consists in defining the path $\sigma_r$.
    Let us reorder the indexes of the eigenvectors in such a way that 
     $\Im(-u_{j})<\Im(-u_{j+1})$ for $j=0, \dots, d-1$.
    Define 
    \begin{equation*}
        \a_j(r):=e^{\ir \beta_j}Z^{\tau}(E_{j})
    \end{equation*}
    for $j=0,\dots, d$ and some $\beta_j\in \R$ to be determined later.
    By assumption we have the following asymptotics
    \begin{equation*}
        \ln{\a_j(r)}\sim \frac{\mathrm{dim}(F)}{2}\ln{(r)}+\Re(-u_j(\tau))/r+\ir\left(\Im(-u_j(\tau))/r+\beta_j\right) \text{ as } r\to 0^+.
    \end{equation*}
    Note that $\a_j$, $j=0,\dots,d$ never vanish because they are non trivial solutions of a linear differential equation whose solutions are uniquely determined by their initial values.
    We define the phases of $E_j$ as  $\phi(E_j):=\frac{1}{\pi}\Im(\ln{\a_j(r)})$.

    Let us fix a $0<\epsilon\ll 1$ such that
    $\frac{\epsilon}{\pi}|\Im(-u_j(\tau))|<\frac{1}{4}$, $j=0,\dots,d$.
    Let us observe that
    $|\frac{\epsilon}{\pi}(\Im(u_{j+1}(\tau))-\Im(u_j(\tau)))|<\frac{1}{2}$ for $j=0,\dots,d-1$.
    Let us fix
    \begin{equation*}
    \begin{split}
        \mu:=\min\left\{\frac{\epsilon}{\pi}(\Im(u_{j+1}(\tau))-\Im(u_j(\tau)))\right\}\\
        \mu':=\min\left\{\frac{\epsilon}{\pi}(\Im(u_{j}(\tau))-\Im(u_{j+1}(\tau)))\right\}
    \end{split}
    \end{equation*}
    and observe that $-\frac{1}{2}<\mu<0<\mu'<\frac{1}{2}$.
    Let us choose a $\delta'>0$ small enough such that $\frac{1}{\delta'}>\epsilon$ and such that for each $r\in (0,2\delta')$
    the error in the estimations of the phases
    \begin{equation*}
        \phi(E_j)\sim \frac{1}{\pi}(\Im(-u_j(\tau))/r+\beta_j) 
    \end{equation*}
    is less then $\min\{\frac{-\mu}{4},\frac{\mu'}{4}\}$.

    Choose $\beta_0$ such that 
    $\frac{1}{\pi}(\Im(-u_{0}(\tau)/\delta')+\beta_{0})=\frac{1}{2}$
    and determine the others $\beta_i$ by the following recursive relations
    \begin{equation*}
        \frac{1}{\pi}\Im(-u_{i+1}(\tau)/\delta')+\beta_{i+1}-\frac{1}{\pi}(\Im(-u_{i}(\tau)/\delta')+\beta_{i})=1.
    \end{equation*}
    see Figure~\ref{fig_asymptotics}.
    Note that the errors in the estimations that we did till now does not change if we vary $\beta_j$, $j=0,\dots,d$.
\begin{figure}[ht]
\includegraphics[width=8cm]{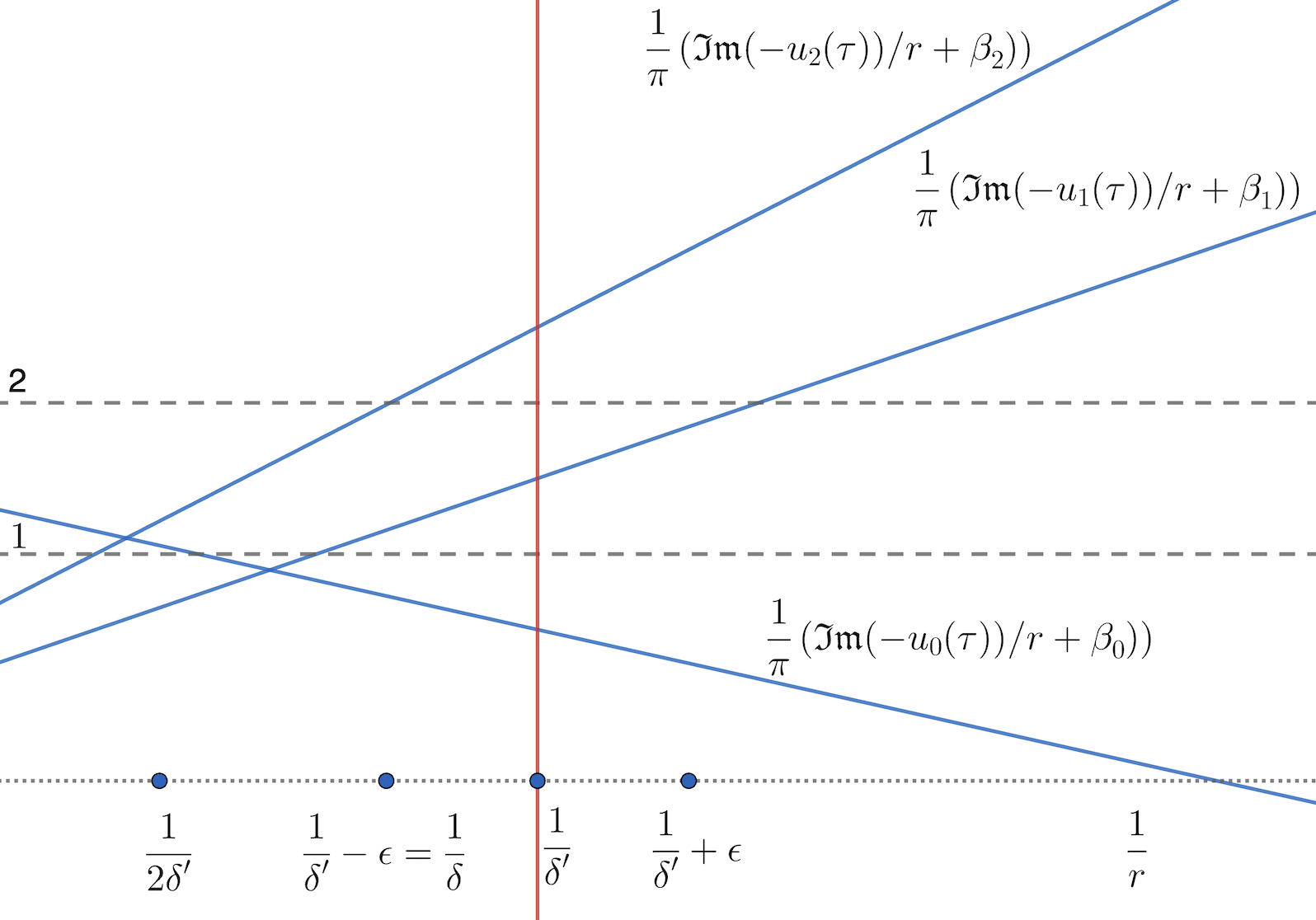}
\caption{Plot of the approximated phases, note that on the horizontal axis we have $\frac{1}{r}$.}
\label{fig_asymptotics}
\end{figure}
    Fix $\delta>0$ such that $\frac{1}{\delta}=\frac{1}{\delta'}-\epsilon$.
    We observe that for $1/r\in (1/\delta'-\epsilon,1/\delta'+\epsilon)$ we have that $\phi(E_j)\in (j,j+1)$:
    for $r=1/\delta'+t\epsilon$, $t\in[-1,1]$ we have 
    \begin{equation*}
    \phi_r(E_j)=j+\frac{1}{2}+t\frac{\epsilon}{\pi}\Im(-u_j(\tau))+\text{error}
    \end{equation*}
    and we also know that $|t\frac{\epsilon}{\pi}\Im(-u_j(\tau))|<\frac{1}{4}$ and $|\text{error}|<\frac{1}{8}$. 
    For $r=\delta$ we have
    \begin{equation*}
        \begin{split}
            |\phi(E_{j+1})-\phi(E_j)|< 
             & \frac{-\mu}{2} +|\frac{1}{\pi}\Im(-u_{j+1}(\tau))/(1/\delta'-\epsilon)+\beta_{j+1}-\frac{1}{\pi}(\Im(-u_j(\tau))/(1/\delta'-\epsilon)+\beta_j)|\\
             <&\frac{-\mu}{2}+|1+\frac{\epsilon}{\pi}(\Im(u_{j+1}(\tau))-\Im(u_{j}(\tau)))|\\
             <&1+\frac{\mu}{2}<1.
        \end{split}
    \end{equation*}
    Similarly for $1/r=1/\delta'+\epsilon$
    we have 
    \begin{equation*}
        |\phi(E_{j+1})-\phi(E_j)|>1+\mu'>1.
    \end{equation*}
    So on the interval $1/r>1/\delta'+\epsilon$ we have $\phi(E_{j+1})-\phi(E_j)>1$.

    We have just proved that for each $r\in(0,\delta)$ the assumptions of Theorem~\ref{thm_stabcond_fromexc} are satisfied. So there is a path of algebraic stability conditions $\sigma_r$, $r\in (0,\delta)$.
    The path is quasi-convergent for $r\to 0$ indeed for $r\ll1$ we have $|\phi(E_{j+1})-\phi(E_j)|>2$ so the only stable objects are $E_j$, $j=0,\dots,d$ and their shifts, see Proposition~\ref{prop_pure_alg_stabcond}.
\end{proof}

\subsection{Projective spaces}\label{subsection_proj_spaces}
In this section we verify that the assumptions of Proposition~\ref{prop_extimated_Ztau} and Theorem~\ref{thm_gammaII_proposalIII} holds for projective spaces: so we fulfill Proposal~\ref{proposalIII_DHL} for all projective spaces.
Moreover for a particular choice of full exceptional collection of $\pr$ we construct a quasi-convergent path of stability conditions that starts in the geometric region, see Corollary~\ref{cor_path_starts_in_geom_proj}.
In Theorem~\ref{thm_conj_pr} we prove Conjecture~\ref{conj_sigma_from_qcoh} for all projective spaces.

\begin{thm}\label{thm_proposalIII_proj}
    Let $\{E_0, \dots, E_{N-1}\}$ be a full, strong exceptional collection of $\mathrm{D^b}(\pr)$. 
    Then for a generic $\tau\in H^2(\pr,\C)$ there exists a quasi-convergent path $\sigma_r=(Z',\slicing_r)$ with 
    \begin{equation*}
        Z'(\alpha)=\int_{\pr}\Phi'(\alpha),\qquad \alpha\in H^{\bullet}(\pr,\C)
    \end{equation*}
    for $\Phi'$ a fundamental solution of $\Tilde{\nabla}_{|_{\tau}}$.
    Moreover the SOD given by $\sigma_{\bullet}$ is $\langle E_0,\dots,E_{N-1}\rangle$.
\end{thm}
\begin{proof}
    Let us consider an idempotent $\Psi_j^{\tau}$ of $\star_{\tau}$, it is in particular an eigenvector of $E_{\tau}$. 
    If by contradiction $\int_{\pr}\Psi_j^{\tau}=0$ then from the equality $E_{\tau}\Psi_j^{\tau}=u_j\Psi_j^{\tau}$ we get that the first entry of the vector $\Psi_j^{\tau}$ is zero. 
    The matrix $E_{\tau}$ behaves as a permutation matrix, see~\eqref{def_E_tau}, so the second entry of $\Psi_j^{\tau}$ is also zero, by recursion we get that $\Psi_j^{\tau}=0$, a contradiction.
    We have verified assumptions of Proposition~\ref{prop_extimated_Ztau}, so we get the asymptotics in~\eqref{eq_5}.
    In order to apply Theorem~\ref{thm_gammaII_proposalIII} we are left to verify that for a generic $\tau\in H^2(\pr, \C)$ the imaginary parts $\Im(-u_j(\tau))$ are distinct. This is clear since $u_j(\tau)=Ne^{-j2\pi \ir /N}e^{\tau/N}$.
\end{proof}

Let us consider the full strong exceptional collection 
    $\{\Omega^i(i)\}$ 
    and let us consider the path of stability conditions $\sigma_r$ constructed in Theorem~\ref{thm_proposalIII_proj}.
\begin{cor}\label{cor_path_starts_in_geom_proj}
    The path $\sigma_r$ associated to the exceptional collection $\{\Omega^i(i)\}$ starts in the geometric part of the space of stability conditions.
\end{cor}
\begin{proof}
    Let us recall that by construction $\Omega^i(i)$ are all stable for all $r\in (0,\delta)$. 
    Moreover for $r\sim \delta$ we have that $\phi(\Omega^{i}(i))-\phi(\Omega^{i+1}(i+1))<1$, see Theorem~\ref{thm_gammaII_proposalIII}.
    We conclude by noticing that we can apply Proposition~\ref{prop_beilinsonexc_geom_stab} so  
for $r\sim \delta$ the skyscraper sheaves of points are stable of the same phase.
We conclude that the path starts in the geometric region.
\end{proof}

Let us discuss Conjecture~\ref{conj_sigma_from_qcoh}: we start with a necessary condition.
Let us consider a quasi-convergent path $\sigma_r$ in $\mathrm{D^b}(\pr)$ whose central charge is given by~\eqref{eq_def_Ztau} i.e., for $V\in \Stab(\mathrm{D^b}(\pr))$
    \begin{equation*}
        Z^{\tau}(V):=(2\pi r)^{\frac{\mathrm{dim}(F)}{2}}\int_{F}\Phi^{\tau}(\hat{\Gamma}\cdot\Ch(V)).
    \end{equation*}
    Let us fix an admissible phase $\phi\in [-\frac{\pi}{2},\frac{\pi}{2}]$ and a full exceptional collection $\{E_j\}$ obtained as in the proof of the Gamma Conjecture II at $\tau\in \H^2(\pr,\C)$.
    Let us denote by $\langle\D_1,\dots,\D_m\rangle$ the SOD given by the path, see Proposition~\ref{prop2.5.1_GGI}.
    By a direct application of the results in Section~\ref{paragraph_general_results} we have that
    \begin{enumerate}
        \item[($\bullet$)] $m$ is the number of distinct imaginary parts of $\{u_j(\tau)\}$,
        \item[($\bullet$)] if $\{E_j\}$ are all limit semistable then $\D_h=\langle E_j| \Im(-u_j(\tau))=\gimel(\D_h) \rangle$.
    \end{enumerate}

\begin{cor}\label{cor_neg_conj_pr}
    Let $\{E_j\}$ be the exceptional collection obtained as mutation of $\{\o,\dots,\o(N-1)\}$ in solution of the Gamma Conjecture II for $\pr$ at $\tau_0=0$.
    Let $Z^{\tau}$ be the quantum central charge~\eqref{eq_def_Ztau} for $\tau$ near $0$.
Then there are no quasi-convergent paths $\sigma_r=(Z^{\tau}, \slicing_r)$ such that $\{E_j\}$ are limit semistable.
\end{cor}
\begin{proof}
    Since $\{\o(i)\}_{i\in\Z}$ is a helix we have in particular that $\Hom^{\bullet}(E_j,E_h)\neq 0$ for $j< h$.
    By construction it is clear that 
    \begin{equation*}
        \ln{Z^{\tau}(E_j)}\sim \frac{N-1}{2} \ln{(2\pi r)}-\zeta^{i(j)}Ne^{\tau/N}/r \text{ as } r\to 0^+ .
    \end{equation*}
    By construction we also have that 
    \begin{equation*}
        \begin{split}
            \{E_j\}&\to \{\Im(-\zeta^{i(j)}Ne^{\tau/N})\}\\
            E_j&\mapsto \Im(-\zeta^{i(j)}Ne^{\tau/N})
        \end{split}
    \end{equation*}
    is not increasing.
    So if $E_j$ are limit semistable Corollary~\ref{cor_increasing} gives us a contradiction, see Figure~\ref{figure_conf_wrong}.
\begin{figure}[ht]
\includegraphics[width=8cm]{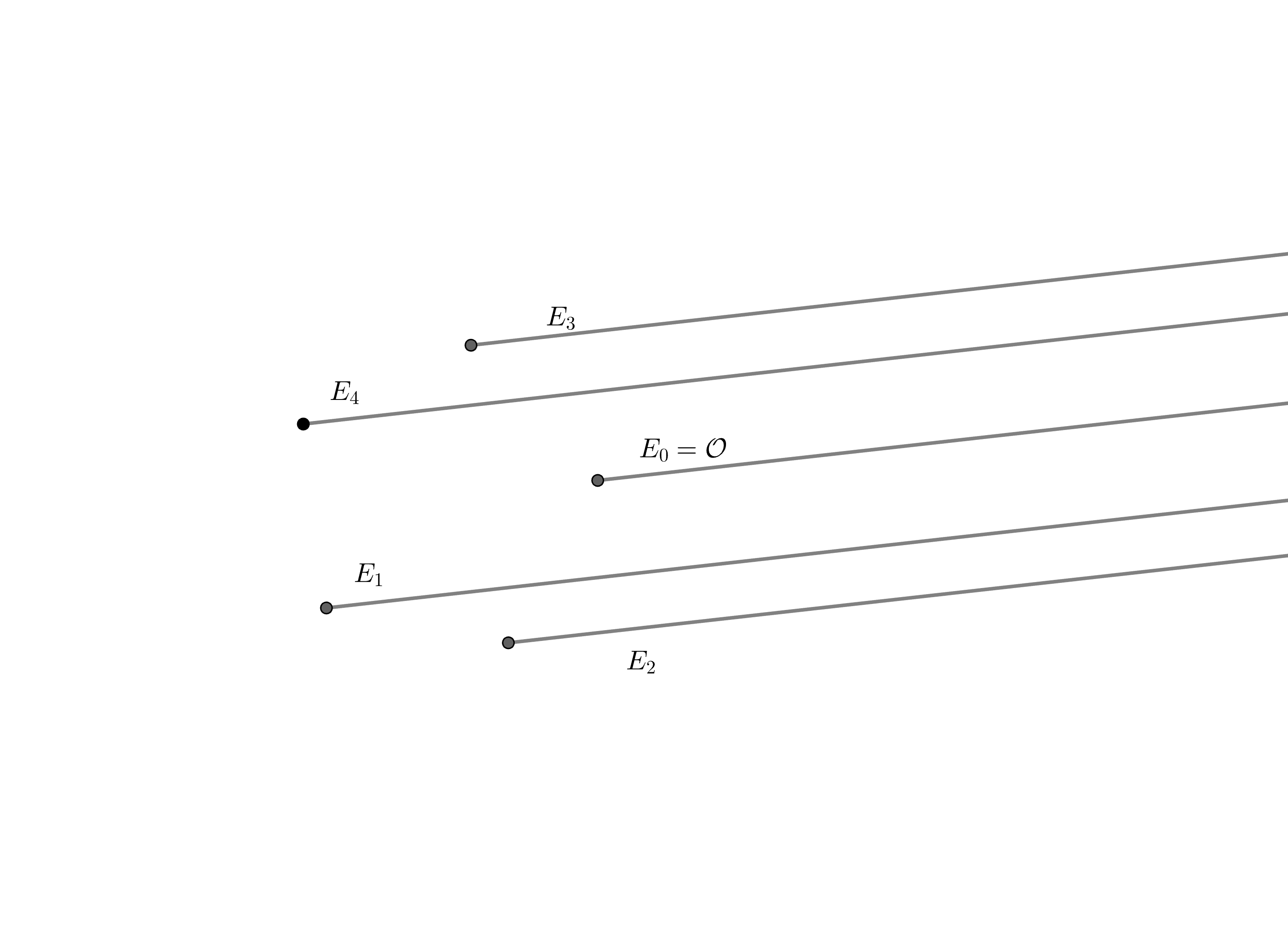}
\caption{Configuration of the integration paths and the associated objects in the exceptional collection.}
\label{figure_conf_wrong}
\end{figure}
\end{proof}
In contrast to the Corollary~\ref{cor_neg_conj_pr} above we have a positive result.

\begin{thm}\label{thm_conj_pr}
    For any $\tau\in\H^2(\pr,\C)$ there is a quasi-convergent path of stability conditions $\sigma_r=(Z^{\tau}, \slicing_r)$ and $\{E_j\}$ a mutation of 
    \begin{equation*}
        \left\{\o\left(\ceil*{\frac{N}{2}}\right),\o\left(\ceil*{\frac{N}{2}}+1\right), \dots, \o\left(\ceil*{\frac{N}{2}}+N-1\right)\right\}
    \end{equation*}
    such that the object $E_j$ is $\sigma_r-$stable for $j=0,\dots,N-1$. Moreover 
    \begin{enumerate}
        \item[\rm{(a)}] if the imaginary parts $\Im(u_j(\tau))$, $j=0, \dots, N-1$ are all distinct then the path $\sigma_r$ gives as SOD (see Proposition~\ref{Proposition_2.20_HLJR}) the exceptional collection $\{E_j\}$,
    \item[\rm{(b)}] if some imaginary parts of 
    $\{-u_j(\tau)\}$ coincide then the components  of the SOD 
    $$\langle \D_{\alpha}|\alpha\in \{\Im(-u_j(\tau))\}\rangle=\mathrm{D^b}(\pr)$$
    given by $\sigma_r$ are indexed by the imaginary values $\{\Im(-u_j(\tau))\}$ and $\D_{\alpha}$ is generated by those $E_j$ whose phase has the asymptotics $\sim \alpha/r$.
    \end{enumerate}
\end{thm}

\begin{proof}
We start by proving the claim for $\tau$ near $0$.
    Let $\{E_j\}$ be the full  exceptional collection obtained by 
    $ \left\{\o\left(\ceil*{\frac{N}{2}}\right),\o\left(\ceil*{\frac{N}{2}}+1\right), \dots, \o\left(\ceil*{\frac{N}{2}}+N-1\right)\right\}$ as in Remark~\ref{rmk_shifted_beil_exc_collection}.
    Then by construction and Proposition~\ref{prop_extimated_Ztau} it is clear that 
    \begin{equation*}
        \ln{Z^{\tau}(E_j)}\sim \frac{N-1}{2} \ln{(2\pi r)}-\zeta^{i(j)}Ne^{\tau/N}/r \text{ as } r\to 0^+ 
    \end{equation*}
    note that we have implicitly chosen an admissible phase $0<\phi\ll1$.
    Let us recall the notation $u_j(\tau):=\zeta^{i(j)}Ne^{\tau/N}$.
    For $\tau$ near $0$ we have that
    \begin{equation*}
        \begin{split}
            \gimel\circ\F:\{E_j\}&\to \{\Im(-\zeta^{i(j)}Ne^{\tau/N})\}\\
            E_j&\mapsto \Im(-\zeta^{i(j)}Ne^{\tau/N})
        \end{split}
    \end{equation*}
    is increasing, see Figure~\ref{figure_conf_correct}.
\begin{figure}[ht]
\includegraphics[width=8cm]{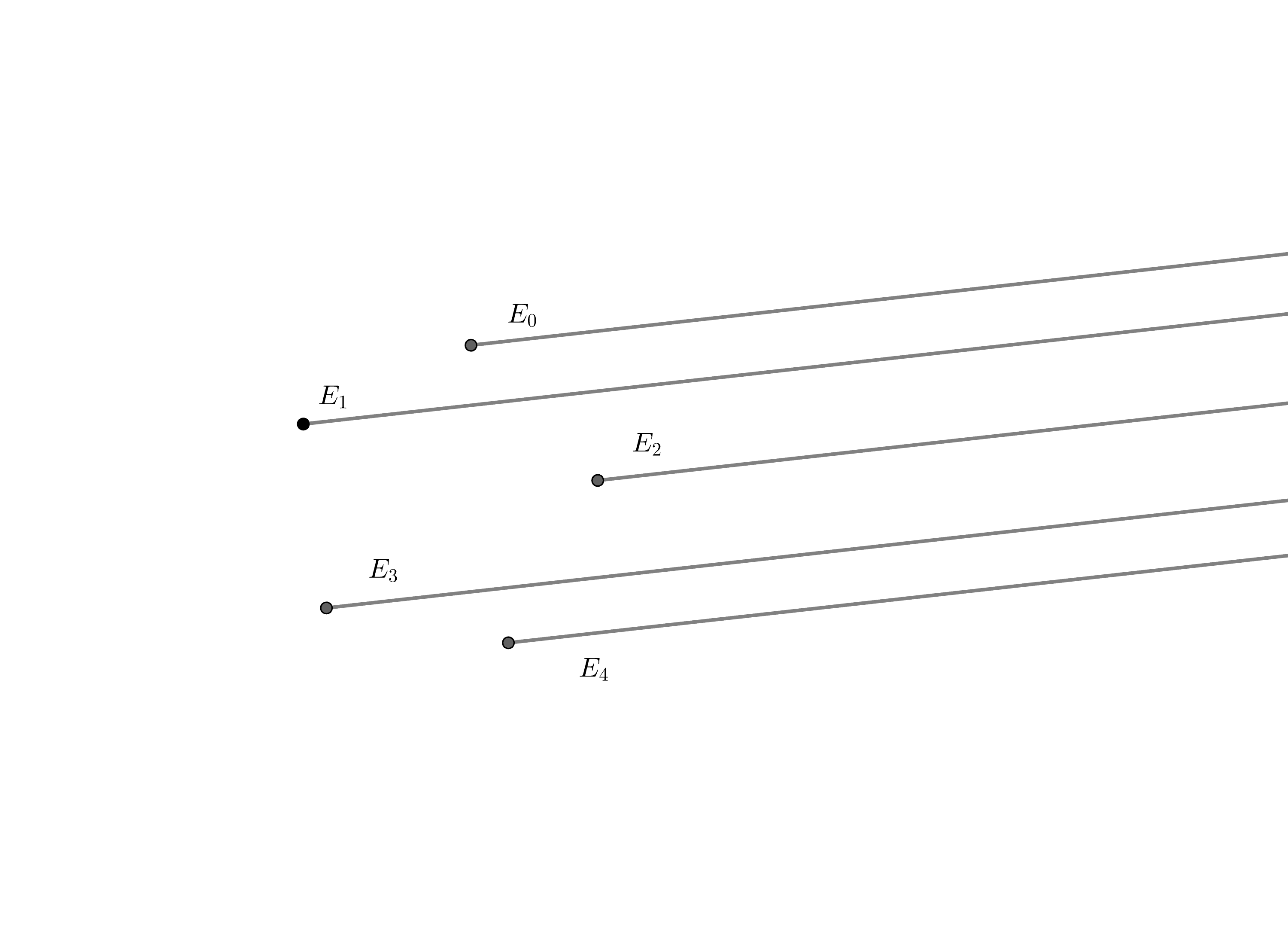}
\caption{Configuration of the integration paths and the associated objects in the exceptional collection.}
\label{figure_conf_correct}
\end{figure}
    For a generic $\tau$ the imaginary values $\Im(-u_j(\tau))$ are all distinct and the map $\gimel\circ\F$ above is strictly increasing.
    For $r\ll 1$ the distance between the phases of $Z^{\tau}(E_{j})$ $j=0,\dots, N-1$ diverges so we can apply Theorem~\ref{thm_stabcond_fromexc} and get a path of algebraic stability conditions $\sigma_r$, $r\in (0,\delta)$.
    By construction and Proposition~\ref{prop_pure_alg_stabcond} we have that for $r\ll 1$ the stability condition  $\sigma_r$ lies in the pure algebraic part of $\Stab(\mathrm{D^b}(\pr))$ associated to the exceptional collection $\{E_j\}$.
    If we start moving $\tau \in \H^{2}(\pr,\C)\cong \C$ from a region near $0$ to any other $\tau'\in \H^2(\pr,\C)$ the exceptional collection will undergo some mutations but the map $\gimel\circ \F$ will remain increasing, see Figure~\ref{figure_1}.
    So for any $\tau\in H^2(\pr,\C)$ such that $\{\Im(-u_j(\tau))\}$ are all distinct we can construct a quasi-convergent path as in  the claim.

    We finish the proof by showing the claim for a $\tau_1\in \H^2(\pr,\C)$ such that some $\Im(-u_j(\tau_1))$, $j=0; \dots, N-1$ may coincide.
    It is clear that there is a path $\tau(t)$, $t\in [0,1]$ with $\tau(1)=\tau_1$ and such that all $t\neq 1$ the imaginary parts 
    $\Im(-u_j(\tau(t)))$ are all distinct.
    So for $t\neq 1$ we have a family of quasi-convergent paths $\sigma_r^t=(Z^{\tau(t)},\slicing_r^t)$.
    By construction we can assume that $\sigma_r^t$, $t\in[0,1)$ are all defined on the same interval $(0,\delta)$.
    For any $r\in (0,\delta)$ fixed we consider the path of stability conditions
    \begin{equation*}
        \begin{split}
            [0,1)&\to \Stab(\mathrm{D^b}(\pr))\\
            t&\mapsto \sigma_r^t
        \end{split}
    \end{equation*}
    there is an associated path of central charges
    \begin{equation*}
        \begin{split}
            [0,1]&\to \Hom(\H^{\bullet}(\pr,\C),\C)\\
            t&\mapsto Z^{\tau(t)}
        \end{split}
    \end{equation*}
    We now prove that $\sigma_r^t$ extends to $t=1$.
    The constant for the support property of the stability conditions $\sigma_r^t$, $t\in[0,1)$ is the same, see Remark~\ref{rmk_alg_stab}.
    Let us denote it by $C:=\frac{C'}{C''}$, see Theorem~\ref{thm_stabcond_fromexc}. 
    In order to prove that the path extends to $t=1$ we apply the deformation theorem \cite[Theorem 1.2]{bayer2019short}.
    So we only need to check that 
    there exists a quadratic form $Q$ on $K(\mathrm{D^b}(\pr))\otimes \R$ such that for some $t\in [0,1)$ 
    \begin{itemize}
        \item[(\rm{1})] $Q$ is negative definite on $\Ker(Z^{\tau(t)})$,
        \item[(\rm{2})] for each  $\sigma_r^t$-semistable object $E$ we have that $Q(E)\geq 0$,
        \item[(\rm{3})] $Q$ is negative definite on $\Ker(Z^{\tau(1)})$.
    \end{itemize}
    It is clear that the quadratic form 
    \begin{equation*}
        Q(\cdot):=C^2|Z^{\tau(t)}(\cdot)|^2-\norm{\cdot}^2
    \end{equation*}
    satisfy the first two conditions.
    We also observe that for $v\in \Ker(Z^{\tau(1)})$ of norm one we have
    \begin{equation*}
        \begin{split}
            Q(v)&=C^2|Z^{\tau(t)}(v)|^2-1\\
            &=C^2(|Z^{\tau(t)}(v)|^2-|Z^{\tau(1)}(v)|^2)-1
        \end{split}
    \end{equation*}
so for any $t\in [0,1)$ such that $\norm{Z^{\tau(t)}}^2-\norm{Z^{\tau(1)}}^2<\frac{1}{2C^2}$ also the last condition is satisfied. 
    It remains to prove that the obtained path $\sigma_r=\sigma^1_r$ is quasi-convergent. 
    We recall that stability is an open condition so $E_j$ are all $\sigma_r$-stable for any $r\in(0,\delta)$, in particular they are limit semistable. 
    The limit semistable filtration of an object $D\in \D$ is the following one: we start with the filtration with respect to  the SOD $\langle E_j\rangle$ and we shrink the filtration when for two consecutive factors $\bigoplus_{i_j}E_j[p_{i_j}],\bigoplus_{i_{j+1}}E_{j+1}[p_{i_{j+1}}]$ we have $\Im(-u_j(\tau_1))=\Im(-u_{j+1}(\tau_1))$.
    For each value $\alpha\in \{\Im(u_j(\tau_1))\}$ we define
    \begin{equation*}
        \D_{\alpha}=\langle E_j|\Im(-u_j(\tau_1))=\alpha\rangle    
    \end{equation*}
    it is a category generated by one or two elements and it consists only of limit semistable objects.
    We conclude that for any limit semistable object $D\in\D_{\alpha}$ 
    \begin{equation*}
        \phi_t(D)\sim \alpha/r \text{ as } r\to 0^+ 
    \end{equation*}
    for some $\alpha\in \{\Im(u_j(\tau_1))\}$.
    This proves that $\sigma_r$ is quasi-convergent.
    Moreover it is clear from the construction that the SOD given by the path is $\langle\D_{\alpha}|\alpha\in\{\Im(-u_j(\tau_1))\}\rangle$.
\begin{figure}[ht]
\includegraphics[width=8cm]{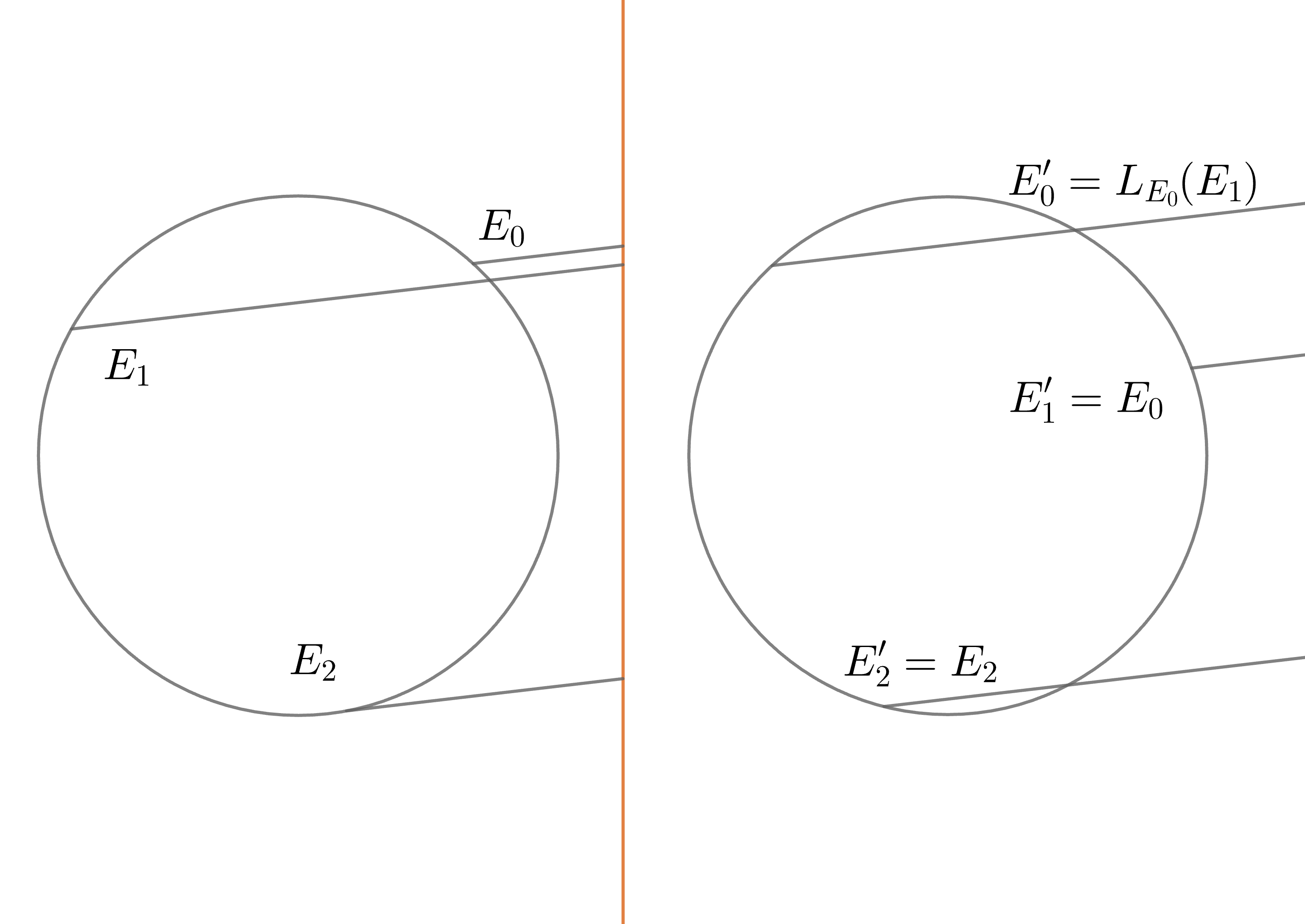}
\caption{Mutation of the exceptional collection caused by the variation of $\tau$.}
\label{figure_1}
\end{figure}
\end{proof}

\subsection{Dimension 2}
For the projective plane we have some stronger and simpler results.
\begin{prop}
    For the projective plane $\pr^2$ and a full strong exceptional collection of coherent sheaves, the quasi-convergent path constructed in the proof of  Theorem~\ref{thm_proposalIII_proj} starts in the geometric region of $\Stab(\mathrm{D^b}(\pr^2))$.
\end{prop}
\begin{proof}
    We follow the same reasoning as in Corollary~\ref{cor_path_starts_in_geom_proj} and we conclude applying Proposition~\ref{prop_geom_alg_stabcond_P2}.
\end{proof}

\begin{thm}\label{thm_conj_P2}
    There exists a quasi-convergent path in $\Stab(\mathrm{D^b}(\pr^2))$ of stability conditions $\sigma_r=(Z^{\tau}, \slicing_r)$, $r\in(0,\delta)$ for some $\delta>0$ and $\tau\in \H^2(\pr^2,\C)$, such that the SOD given  at $r\to 0^+$ is $\{\o,\o(1),\o(2)\}$.
\end{thm}
\begin{proof}
    By Proposition~\ref{prop_extimated_Ztau} we have that 
    \begin{equation*}
        Z^{0}(\o(j))\sim 2\pi z e^{-u_j(0)/z} \text{ as }z\to 0^+ \text{ in the sector } |\arg(z)-\phi|<\frac{\pi}{2}+\epsilon
    \end{equation*}
    where $0<\phi\ll 1$ is an admissible phase.
    
    We now change admissible phase and choose $\phi'=\frac{\pi}{2}-\frac{1}{100}$, during the move the integration ray starting at $u_1$ crosses $u_0$ from the right.
    Then we move $\tau$ form $0$ to $3(\frac{\pi}{2}+\frac{1}{100})\ir$ on the path $\tau(t):=t3(\frac{\pi}{2}+\frac{1}{100})\ir$, $t\in[0,1]$ and $u_j$ transform as follows
    $$u_j(\tau(t))=3e^{-j2\pi\ir/N }e^{t(\frac{\pi}{2}+\frac{1}{100})\ir}$$
    for $t\in[0,1]$.
    
    During the move  $u_0$ crosses the ray of $u_1$ from the right.
    Ad the end of the two moves the exceptional collection undergoes two mutations that compensate each other, so we get that the asymptotically exponential solution at $\tau(1)$ is associated to the exceptional collection 
    \begin{equation*}
        \{\o,\o(1),\o(2)\}
    \end{equation*}
    and we have
    \begin{equation*}
        \ln{Z^{\tau(1)}(\o(j))}\sim \frac{\mathrm{dim}(F)}{2} \ln{(2\pi z)}-u_j(\tau(1))/z \text{ as } z\to 0 \text{ in the sector } |\mathrm{arg}(z)-\phi'|<\frac{\pi}{2}+\epsilon
    \end{equation*}
      Clearly we have 
    $$\Im(u_0(\tau(1)))>\Im(u_1(\tau(1)))>\Im(u_2(\tau(1)))$$
    At this point we proceed as in the proof of Theorem~\ref{thm_gammaII_proposalIII} and we get a path defined for $z=r\in (0,\delta)$ as in the  claim.
\end{proof}

\end{document}